\documentclass[11pt,A4]{amsart}

\usepackage{amsmath,amssymb,amsfonts,epsfig,mathrsfs,cite, hyperref}
\usepackage[T1]{fontenc}
\usepackage[utf8]{inputenc}
\usepackage{color}
\usepackage{array}
\usepackage{amsthm}
\usepackage{amstext}
\usepackage{graphicx}
\usepackage{setspace}

\usepackage[title]{appendix}

\usepackage{booktabs}

\usepackage{tcolorbox}

\usepackage{float}

\makeatletter
\@namedef{subjclassname@2020}{%
	\textup{2020} Mathematics Subject Classification}
\makeatother

\usepackage[margin=2.5cm]{geometry}
\usepackage{color}
\usepackage{enumitem}
\usepackage{amscd,psfrag}
\usepackage{yhmath}
\usepackage[mathscr]{eucal}
\usepackage{comment}

\allowdisplaybreaks[4]

\makeatletter
\pdfpageheight\paperheight
\pdfpagewidth\paperwidth

\usepackage{epstopdf}
\usepackage{indentfirst}	

\usepackage[normalem]{ulem}
\theoremstyle{plain}
\newtheorem{definition}{Definition}
\newtheorem{theorem}[definition]{Theorem}
\newtheorem*{theorem*}{Theorem}

\newtheorem*{remark*}{Remark}
\newtheorem*{sideremark*}{Side Remark}

\newtheorem*{claim*}{Claim}
\newtheorem*{lemma*}{Lemma}
\newtheorem*{q*}{Question}
\newtheorem{lemma}[definition]{Lemma}

\newtheorem*{corollary*}{Corollary}

\newtheorem{proposition}[definition]{Proposition}

\newcommand{\R}{\mathbb{R}}

\newcommand{\emb}{\hookrightarrow}

\newcommand{\dd}{{\rm d}}

\def\XXint#1#2#3{{\setbox0=\hbox{$#1{#2#3}{\int}$ }
		\vcenter{\hbox{$#2#3$ }}\kern-.6\wd0}}



\date{\today}

\author{Siran Li}

\address{Siran Li: School of Mathematical Sciences $\&$ CMA-Shanghai, Shanghai Jiao Tong University, No.~6 Science Buildings,
800 Dongchuan Road, Minhang District, Shanghai, China (200240)}

\email{\texttt{siran.li@sjtu.edu.cn}}

\author{Jianing Yang}

\address{Jianing Yang: School of Mathematical Sciences, Shanghai Jiao Tong University, No.~6 Science Buildings,
800 Dongchuan Road, Minhang District, Shanghai, China (200240)}

\email{\texttt{jnyang22@sjtu.edu.cn}}

\author{Yuantu Zhu}

\address{Yuantu Zhu: School of Mathematical Sciences, Shanghai Jiao Tong University, No.~6 Science Buildings,
800 Dongchuan Road, Minhang District, Shanghai, China (200240)}

\email{\texttt{radonzhu@sjtu.edu.cn}}

\begin{document}
	\title{Existence of Weak Solutions to a Power-Law Model for Compressible Non-Newtonian Fluids on 1D Unbounded Domain}
	
	\begin{abstract}
This paper is concerned with the analysis of a one-dimensional power-law model for compressible fluid dynamics on $\mathbb{R}$, in which the shear stress takes the form $\mu |\partial_{x}u|^{p-2}\partial_{x}u$, where $\mu$ is the viscosity coefficient and $u$ is the velocity. We prove that, in the singular limit $p\rightarrow\infty$, the solutions converge to functions $(\rho,u)$ satisfying $|\partial_{x}u|\leq 1$, $\tau = \pi \partial_{x}u$, $\pi \geq 0$, and $\pi (1 - |\partial_{x}u|) = 0$ a.e. on $\R$. Moreover, we rigorously justify the existence of weak solutions to the  limiting equation. The convergence as $p \to \infty$ is obtained via domain truncation and compactness arguments, of which the key challenge is to show that the density remains bounded away from zero and infinity on any compact subset. This extends the recent result of Bresch, Burtea, and Szlenk [Nonlinearity 26 (2026), no. 5, Paper No. 055010.] from one-dimensional periodic domain to the whole real line.
	\end{abstract}
	\maketitle
	
	\section{Introduction}\label{sec: intro}
		Non-Newtonian fluids arise in a wide range of industrial and biological applications, including blood, molten plastics, printing inks, and food pastes; see, \emph{e.g.}, \cite{Chhabra2008,Chhabra2010}. In contrast to Newtonian fluids, such as air or water, for which the shear stress is linear in the velocity gradient, the complex fluids exhibit nonlinear relations between the stress and the rate of strain. 
	
	Among the extensive non-Newtonian fluid models, the \emph{power-law model} is of particular interest in theory and applications. The $p^{\text{th}}$-power-law model, where $p \in [1,\infty[$ is a given parameter, for one-dimensional compressible non-Newtonian fluids takes the following form:
	\begin{equation}\label{p-system}
		\begin{cases}
			\partial_t \rho_{p} + \partial_x (\rho_{p} u_{p}) = 0,\\
			\partial_t (\rho_{p} u_{p}) + \partial_x (\rho_{p} u_{p}^2) - \mu \partial_x\tau_p + a \partial_x \rho_{p}^\gamma = 0,
		\end{cases}
	\end{equation} 
	Here, $\rho_{p},u_{p}$ are the density and the velocity of the fluid, respectively; $a > 0$ is a constant related to pressure, and $\gamma > 1$ is the adiabatic exponent. The shear stress $\tau_{p}$ is given by
	\begin{equation}\label{ShearStress}
		\tau_p = |\partial_x u_p|^{p-2}\partial_x u_p,
	\end{equation}
	where $\mu > 0$ is the viscosity coefficient and $p$ is the power-law index. The exponent $p$ distinguishes between different rheological regimes: $p > 2$ characterises shear-thickening fluids (\emph{i.e.}, viscosity increases with increasing shear rate), $p < 2$ shear-thinning fluids (\emph{i.e.},  viscosity decreases with increasing shear rate), and $p = 2$ the Newtonian fluid.
	
	The mathematical theory of power-law fluids has been a crucial topic of research in the literature of fluid dynamics and partial differential equations (PDE). The analytic theory of incompressible non-Newtonian fluids has been investigated in the pioneering work of Ladyzhenskaya~\cite{Ladyzhenskaya1967} and many important subsequent developments~\cite{Bothe2007,Berselli2010,Guo2002,Malek1993,Malek2001,Moscariello2020,Wolf2007}. For compressible power-law non-Newtonian fluids, various results on global existence, large-time behaviour, and decay rates have also been established. Yuan--Si--Feng \cite{Yuan2019} proved the existence of global strong solutions to a class of compressible non-Newtonian fluids with small initial energy. Feireisl--Kwon--Novotny \cite{Feireisl2021} analysed the large-time behaviour of dissipative solutions to compressible non-Newtonian models. In the setting of one space dimension, Yuan--Liu--Qiao \cite{Yuan2012} established the global existence of strong solutions for isentropic compressible non-Newtonian fluids. More results on non-Newtonian fluids can be found in \cite{Fang2023,Lian2012,Muhammad2020,Zhao1995,Zhao2001} and the many references cited therein.
	
	The study of the asymptotic limit $p\rightarrow\infty$ in the power-law model~\eqref{p-system} $\&$ \eqref{ShearStress} is motivated by the extreme shear-thickening behaviour observed in various complex fluids, such as concentrated suspensions of china clay, titanium dioxide~\cite{Chhabra2008}. Formally, as $p\rightarrow\infty$, the viscosity becomes infinitely sensitive to the velocity gradient, and the stress $\tau_{p}$ tends to $0$ for $|\partial_{x}u|<1$ and infinite for $|\partial_{x}u|>1$. Therefore, a ``\emph{saturation constraint}'' must be imposed on the velocity gradient:
	\begin{equation}\label{cdt}
		\begin{cases}
			|\partial_{x}u|\leq 1 \quad \text{a.e.},\\ \pi \geq 0 \quad \text{with}\quad \pi (1 - |\partial_{x}u|) = 0
		\end{cases}
	\end{equation}
	where $\pi$ acts as a Lagrange multiplier. 
	
	In the recent nice work~\cite{Bresch2026}, Bresch--Burtea--Szlenk proved (among other multi-dimensional results) the global existence of weak solutions to~\eqref{p-system} on the 1-dimensional periodic domain $\mathbb{T}=\mathbb{R} / \mathbb{Z}$ and justified that, as $p \to \infty$, the weak solutions converge to those of the limiting system
	\begin{equation*}
		\begin{cases}
			\partial_{t}\rho +\partial_{x}(\rho u) = 0,\\
			\partial_{t}(\rho u) + \partial_{x}(\rho u^{2}) - \mu \partial_{x}\tau + a\partial_{x}\rho^{\gamma} = 0,
		\end{cases}
	\end{equation*}
	subject to the saturation constraint~\eqref{cdt}. 
	
	The main goal of this note is to establish the existence of weak solutions to the limiting system (formally, with $p=\infty$) on the whole real line $\mathbb{R}$; see \eqref{PDE limit} below. This is achieved by a truncation argument: we first analyse the system~\eqref{p-system} and \eqref{ShearStress} in a finite interval of length $\sim k$ subject to the homogeneous Dirichlet boundary condition, obtaining estimates for the weak solutions independently of the power-law index $p$, and then send $k \to \infty$ to construct a global weak solution on $\mathbb{R}$ by compactness and diagonalisation arguments. 
	
	Here and hereafter, for each $k = 1,2,3,\ldots$, denote the bounded interval $$\Omega_{k} = (-2k-2,2k+2).$$ 
    
    The main result of our paper is as follows.
	\begin{theorem}\label{thm: main}
		Let $(\rho_0, u_0):  \mathbb{R} \to  \mathbb{R}^2$ be the initial data satisfying the following conditions: given any compact $K \Subset \mathbb{R}$, there are constants $c_1(K) > 0$ depending only on $K$ and $c_2 > 0$ universal, such that
		\begin{equation*}
			\begin{cases}    
				0 < c_{1}(K)\leq \rho_{0}(x)\leq c_{2}< \infty \quad \text{for a.e. } x\in K,\\
				\|\partial_{x}u_{0}\|_{L_{x}^{\infty}(\mathbb{R})}< 1,\qquad \| u_{0}\|_{L_{x}^{\infty}(\mathbb{R})}\leq 1,\\
				E_{0}:= \int_{\mathbb{R}}\left(\frac{1}{2}\rho_{0}u_{0}^{2} + \frac{a}{\gamma - 1}\rho_{0}^{\gamma}\right)\mathrm{d}x < \infty.
			\end{cases}
		\end{equation*}
		
		For each $k=1,2,3,\ldots,$ let $(\rho_{k,0}, u_{k,0})$ be the truncated initial data and $(\rho_{p,k,0}, u_{p,k,0})$ be smooth approximations of $(\rho_{k,0}, u_{k,0})$ such that
        \begin{align*}
			&0 < c_1(\Omega_k) \leq \rho_{p,k,0}(x) \leq c_2 < \infty,\qquad \|\partial_x u_{p,k,0}\|_{L_x^\infty(\Omega_k)} \leq 1,\\
			&\rho_{p,k,0}\longrightarrow \rho_{k,0},\quad u_{p,k,0}\longrightarrow u_{k,0}\quad \text{strongly in} \quad L_x^{\infty}(\mathbb{R}).
		\end{align*}
		where $c_2 > 0$ is independent of $p$ and $k$ and  $c_1(\Omega_k)>0$ is dependent on $k$ but independent of $p$. Then, for $p$ sufficiently large, there exist weak solutions $(\rho_{p,k}, u_{p,k})$ to the truncated system
		\begin{equation}
			\label{PDE,pk}
			\begin{cases}
				\partial_{t}\rho_{p,k} + \partial_{x}(\rho_{p,k}u_{p,k}) = 0,\\
				\partial_{t}(\rho_{p,k}u_{p,k}) + \partial_{x}(\rho_{p,k}u_{p,k}^{2}) - \mu \partial_{x}(|\partial_{x}u_{p,k}|^{p-2}\partial_{x}u_{p,k}) + a\partial_{x}\rho_{p,k}^{\gamma} = 0 \quad \text{in } (0,T)\times\Omega_k,\\
				u_{p,k} = 0\qquad \text{ on } \partial\Omega_k.
			\end{cases}
		\end{equation}
        As $p\to \infty$, there exists a subsequence such that for every $1\leq r<\infty$, it holds that
		\begin{align*}
			u_{p,k} &\rightharpoonup u_k \quad \text{weakly in } L^2(0,T; H^1(\Omega_k))\cap L^r(0,T;W^{1,r}(\Omega_k)),\\
			u_{p,k} &\longrightarrow u_k \quad \text{strongly in } L^2((0,T)\times\Omega_k),\\
			\tau_{p,k} := |\partial_x u_{p,k}|^{p-2}\partial_x u_{p,k} &\rightharpoonup \tau_k \quad \text{weakly in } L^2((0,T)\times\Omega_k),\\
			\rho_{p,k} &\longrightarrow \rho_k \quad \text{strongly in } C([0,T]; L^r(\Omega_k)).
		\end{align*}
        Moreover, $(\rho_k, u_k)$ obtained above satisfies the system
        \begin{equation}\label{PDE,k}
			\left\{
			\begin{array}{ll}
				\partial_{t}\rho_k +\partial_{x}(\rho_k u_k) = 0,\\
				\partial_{t}(\rho_k u_k) + \partial_{x}(\rho_k u_k^{2}) - \mu \partial_{x}\tau_k + a\partial_{x}\rho_k^{\gamma} = 0 \qquad \text{in } \mathcal{D}'((0,T)\times \Omega_k);\\
				\tau_k = \pi_k \partial_{x}u_k,\quad \pi_k \geq 0,\quad \pi_k (1 - |\partial_{x}u_k|) = 0, \quad \text{and}\quad 
				|\partial_{x}u_k|\leq 1 \qquad \text{a.e. in } (0,T)\times \Omega_k,\\
                u_k = 0 \qquad \text{on } \partial\,\Omega_k.
			\end{array}
			\right.
		\end{equation}
		Moreover, there exist a subsequence with respect to $k$ (unrelabelled), a pair $(\rho,u):  \mathbb{R}\to  \mathbb{R}^2$, a Radon measure $\tau$ on $(0,T)\times \mathbb{R}$, and a measurable function $\pi \geq 0$ such that as $k\to\infty$,
		\begin{equation*}
			\begin{cases}
				u_k &\rightharpoonup u \quad \text{weakly in } L^2(0,T; H^1_{\mathrm{loc}}(\mathbb{R})),\\
				u_k &\rightarrow u \quad \text{strongly in }  L^2(0,T;L^2_{\mathrm{loc}}(\mathbb{R})),\\
				\tau_k &\stackrel{*}{\rightharpoonup} \tau \quad \text{as Radon measures on } (0,T)\times \mathbb{R}, \\
				\rho_k &\rightarrow \rho \quad \text{strongly in } C([0,T];L^r_{\mathrm{loc}}(\mathbb{R})), \quad r\in[1,\infty[.
			\end{cases}
		\end{equation*}
		The limit $(\rho, u)$ is a weak solution to the system: 
		\begin{equation}\label{PDE limit}
			\left\{
			\begin{array}{ll}
				\partial_{t}\rho +\partial_{x}(\rho u) = 0,\\
				\partial_{t}(\rho u) + \partial_{x}(\rho u^{2}) - \mu \partial_{x}\tau + a\partial_{x}\rho^{\gamma} = 0 \qquad \text{in } \mathcal{D}'((0,T)\times \mathbb{R});\\
				\tau = \pi \partial_{x}u,\quad \pi \geq 0,\quad \pi (1 - |\partial_{x}u|) = 0, \quad \text{and}\quad 
				|\partial_{x}u|\leq 1 \qquad \text{a.e. in } (0,T)\times \mathbb{R}.
			\end{array}
			\right.
		\end{equation}
		Furthermore, for each compact set $K \Subset \mathbb{R}$, we have
		\begin{equation}
			0< C_{1}(K)\leq \rho (t,x)\leq C_{2}(K)<\infty\qquad \text{for a.e. } (t,x)\in (0,T)\times K,
		\end{equation}
		with $C_{1}(K)$ and $C_{2}(K)$ depending only on $T$, $\gamma$, $a$, $\mu$, $E_{0}$, and $K$.
	\end{theorem}
	
	Compared with the periodic setting in Bresch--Burtea--Szlenk~\cite{Bresch2026}, the unbounded domain in our case introduces additional difficulties: the limits $p\to\infty$ and $k\to\infty$ do not commute, and the Dirichlet boundary conditions require new estimates that are uniform in both $p$ and the truncation parameter $k$. These difficulties are overcome in Sections 2--4 through a boundary maximum principle, a modified Basov--Shelukhin potential estimate~\cite{Basov1999}, and a diagonalised compactness argument.
	
	The rest of the paper is organised as follows. In Section~\ref{sec: prelim}, we collect the notation used throughout and recall the Aubin--Lions--Simon compactness lemma. In Section~\ref{sec: unif est for truncated}, we establish $p$-independent estimates for the truncated system~\eqref{PDE,pk}  on $\Omega_{k}$. In Section~\ref{sec: pass to lim}, we pass to the limit $p\to\infty$ for fixed $k$ by compactness arguments. In Section~\ref{sec: proof of thm main}, we complete the proof of Theorem~\ref{thm: main} by taking $k\to\infty$. Finally, in Appendix~\ref{sec: appendix}, we present detailed computations of the \emph{a priori} estimates adapted from \cite{Bresch2026} to the setting of Dirichlet boundary value problem.

	\section{Preliminaries}\label{sec: prelim}
	
	We collect here the notation used throughout the paper.
	\subsection{Notation}
	For $\Omega \subset \mathbb{R}$ and $1\leq r\leq \infty$, $L_{x}^{r}(\Omega)$ denotes the usual Lebesgue space with norm $\| \cdot \|_{L_{x}^{r}(\Omega)}$. For $1\leq r,q\leq \infty$, we write
	\begin{equation*}
		\|f\|_{L^q(0,T;L^r(\Omega))}
		= \left(\int_0^T \|f(t,\cdot)\|_{L_x^r(\Omega)}^q\,\mathrm{d}t\right)^{1/q},
	\end{equation*}
	with the usual modification when $q=\infty$.
	The local Lebesgue spaces $L_{\mathrm{loc}}^{r}(\Omega)$ and local Sobolev spaces $H_{\mathrm{loc}}^{1}(\Omega)$ are defined in the standard way:
	\begin{equation*}
		L^r_{\mathrm{loc}}(U)=\left\{f: U\rightarrow \mathbb R \,\Big|\, \forall K\Subset U,\ \|f\|_{L^r(K)}<\infty\right\},
	\end{equation*}
	 $H^1_{\mathrm{loc}}(U)$ is defined analogously.

	We shall also use the space $C([0,T];L_{w}^{r}(\Omega))$ of weakly continuous functions from $[0,T]$ into $L^{r}(\Omega)$, i.e., $t\mapsto \langle f(t),\phi \rangle$ is continuous for every $\phi \in L^{r'}(\Omega)$, where $r'$ is the conjugate index of $r$.

	For a velocity field $u$, the material derivative is denoted as
	\begin{equation*}
		\dot{f}:= \frac{\mathrm{D}}{\mathrm{D}t} f = \partial_t f + u\partial_x f.
	\end{equation*}
	We write $\mathbb{I}_{X}$ for the indicator function of a set $X\subset (0,T)\times \mathbb{R}$. Also write $C = C(a_{1},\ldots,a_{n})$ for positive constants depending only on the parameters $a_{1},\ldots,a_{n}$, which may change from line to line.
	
	The following lemma (see \cite[Theorem 5]{Simon1987}) will be used when passing to the weak limits of approximate solutions.
	
	\begin{theorem}[Aubin--Lions--Simon]\label{thm:ALS}
		Let $X$, $Y$, and $Z$ be Banach spaces. Assume that $X$ embeds compactly into $Y$, and $Y$ embeds continuously into $Z$. Let $1\leq p\leq \infty$. Assume that for a bounded family of functions $\mathcal{F}\subset  L^{p}(0,T;X)$, one has $\lim_{h \to 0^+}\| f(t + h) - f(t)\|_{L^{p}(0,T-h;Z)}= 0$ uniformly in $f \in \mathcal{F}$. Then $\mathcal{F}$ is relatively compact in $L^{p}(0,T;Y)$ and in $C(0,T;Y)$ if $p = \infty$.
	\end{theorem}
%	
%	\begin{theorem}[Aubin--Lions--Simon compactness]\label{thm:ALS}
%		Assume $X\subset Y\subset Z$ with compact embedding $X\hookrightarrow Y$ and $X$, $Y$, $Z$ are Banach spaces. Let $1\leq p\leq \infty$ and
%		\begin{equation*}
%			\begin{array}{l}
%				F \quad \text{is bounded in }L^p(0,T;X), \\ 
%				\left\|f(t+h)-f(t)\right\|_{L^p(0,T-h;Z)}\rightarrow 0\quad \text{as }h\rightarrow 0,\quad \text{uniformly for }f\in F.
%			\end{array}
%		\end{equation*}
%		Then F is relatively compact in $L^p(0,T;Y)$ and in $C(0,T;Y)$ if $p=\infty$.
%	\end{theorem}
	
	\section{Uniform estimates on truncated domain} \label{sec: unif est for truncated}
	In this section, we establish the $p$-independent estimates needed for the compactness argument on the truncated domain. The main novelty here, compared with the periodic setting, is the Dirichlet boundary condition, which requires a separate analysis of boundary maximum and a modified Basov--Shelukhin potential adapted to the truncated domain.
	
	Let $\eta\in C_c^\infty(\mathbb{R})$ be a standard cut-off function satisfying $0\leq \eta\leq 1$, $\eta=1$ on $[-1,1]$, $\eta=0$ outside $[-2,2]$, and set $C_\eta:=\|\eta^{\prime}\|_{L^\infty}$. For each $k\geq 1$, recall $\Omega_k:=(-2k-2,\,2k+2)$ and define
	\begin{equation*}
	  \phi_k(x):=\eta \left(\frac{x}{k+1}\right).
	\end{equation*}
	Then $\phi_k\in C_c^\infty(\Omega_k)$, $0\leq \phi_k\leq 1$, $\phi_k=1$ on $[-k-1,k+1]$, and $\|\phi_k^{\prime}\|_{L^\infty}\leq C_\eta/(k+1)$.
	The truncated initial data are defined by
	\begin{equation}\label{truncated i.d.}
		\begin{cases}
			\rho_{k,0}(x)=\rho_0(x)\phi_k(x)+\rho_{\star,k}(1-\phi_k(x)),\\
			u_{k,0}(x)=u_0(x)\phi_k(x),
		\end{cases}
	\end{equation}
	where $\rho_{\star,k}=\operatorname{ess}\inf_{x\in\Omega_k}\rho_0>0$.
	Assuming $\|\partial_x u_0\|_{L_x^\infty}<1$, for $k$ sufficiently large we have
	\begin{equation*}
		\|\partial_x u_{k,0}\|_{L_x^\infty(\Omega_k)}
		\leq \|\partial_x u_0\|_{L_x^\infty} + \|u_0\|_{L_x^\infty}\frac{C_\eta}{k+1} \leq 1.
	\end{equation*}
	
	The local existence of strong solutions can be achieved using a classical Galerkin method or fixed point theorem applied to the Lagrangian formulation. We omit the details and refer to \cite{Kalousek2021}. The main point is to derive appropriate uniform bounds with respect to $p$ to define weak solutions for fixed $k$ and then pass to the limit with $p\rightarrow \infty$.
	
	\begin{theorem}
		\label{thm: local}
		For each $k =1,2,3,\ldots,$ let $(\rho_{k,0}, u_{k,0})$ be as in \eqref{truncated i.d.}. Suppose that the initial data $(\rho_{p,k,0},u_{p,k,0})\in L_{x}^{\infty}(\mathbb{R})\times W_{x}^{1,\infty}(\mathbb{R})$ satisfy
		\begin{align*}
			&0 < c_1(\Omega_k) \leq \rho_{p,k,0}(x) \leq c_2 < \infty,\qquad \|\partial_x u_{p,k,0}\|_{L_x^\infty(\Omega_k)} \leq 1,\\
			&\rho_{p,k,0}\longrightarrow \rho_{k,0},\quad u_{p,k,0}\longrightarrow u_{k,0}\quad \text{strongly in} \quad L_x^{\infty}(\mathbb{R}).
		\end{align*}
		where $c_2 > 0$ is independent of $p$ and $k$ and  $c_1(\Omega_k)>0$ is dependent on $k$ but independent of $p$. Fix any sufficiently large $k$ so that the construction of the truncated initial data yields $\|\partial_x u_{k,0}\|_{L_x^\infty(\Omega_k)}\leq 1$. Then, for each sufficiently large $p$ the truncated system 
		\eqref{PDE,pk} together with the Dirichlet boundary condition 
		$u_{p,k}|_{\partial\Omega_k}=0$ admits a weak solution $(\rho_{p,k}, u_{p,k})$.   	
		Furthermore, for any $T>0$, there exists a constant 
		$C = C(T,\gamma,\mu,a,E_0,k)$ independent of $p$ such that
		\begin{align*}
			&\sup_{t\in[0,T]} \|\rho_{p,k}(t)\|_{L_x^\infty(\Omega_k)} \leq C,\quad \int_0^T \int_{\Omega_k} |\partial_x u_{p,k}|^p\,\mathrm{d}x \,\mathrm{d}t \leq C,\\
			&\int_0^T \int_{\Omega_k} \rho_{p,k} |\dot u_{p,k}|^2\,\mathrm{d}x\,\mathrm{d}t \leq C,\quad \|\dot{u}_{p,k}\|_{L^2(0,T;L^2(\Omega_k))} \leq C.
		\end{align*}
		In the above, recall the material derivative: $\dot{u}_{p,k}:= \partial_{t}u_{p,k} + u_{p,k}\partial_{x}u_{p,k}$. Also, note that the initial energy $E_{0} = \int_{\mathbb{R}}\left(\frac{1}{2}\rho_{0}u_{0}^{2} + \frac{a}{\gamma-1}\rho_{0}^{\gamma}\right)\mathrm{d}x$ is independent of $p$ and $k$.
	\end{theorem}
	
	\subsection{Energy estimates}
	The following proposition provides the energy inequality which will be used frequently throughout the paper.
	\begin{proposition}[energy estimates]\label{prop: energy}
		Under the assumptions of $\rho_{p,k,0}$ and $u_{p,k,0}$ stated in Theorem \ref{thm: local}, suppose that $(\rho_{p,k},u_{p,k})$ is the solution of Eq.~\eqref{PDE,pk}. Then the following energy estimate holds
		\begin{equation}\label{es: energy pk}
			\int_{\Omega_k}\left[\frac{1}{2}\rho_{p,k}\left|u_{p,k}\right|^2+\frac{a}{\gamma-1}\rho_{p,k}^{\gamma}\right](x,t)\,\mathrm{d}x+\mu\int_0^t\int_{\Omega_k}\left|\partial_x u_{p,k}\right|^p\,\mathrm{d}x\,\dd s= E_{p,k,0},
		\end{equation}
		where
		\begin{equation*}
			E_{p,k,0}=\int_{\Omega_k}\left[\frac{1}{2}\rho_{p,k,0}\left|u_{p,k,0}\right|^2+\frac{a}{\gamma-1}\rho_{p,k,0}^{\gamma}\right]\,\mathrm{d}x.
		\end{equation*}
		Let $E_{k,0}=\int_{\Omega_k}\left[\frac{1}{2}\rho_{k,0}\left|u_{k,0}\right|^2+\frac{a}{\gamma-1}\rho_{k,0}^{\gamma}\right]\,\mathrm{d}x$ denote the energy of the truncated initial data, and $E_0:=\int_{\mathbb{R}}\left(\frac{1}{2}\rho_0 u_0^2+\frac{a}{\gamma-1}\rho_0^\gamma\right)\,\mathrm{d}x$ the total initial energy. Then 
		\begin{equation*}
			E_{p,k,0}\leq E_{k,0}\leq E_0.
		\end{equation*}
	\end{proposition}
	\begin{proof}
		Multiplying both sides of the second equation in \eqref{PDE,pk} by $u_{p,k}$ and integrating by parts, we obtain
		\begin{equation*}
			\frac{\mathrm{d}}{\mathrm{d}t}\int_{\Omega_k}  \rho_{p,k}\frac{|u_{p,k}|^2}{2}\,\mathrm{d}x +\int_{\Omega_k}\left[\mu|\partial_x u_{p,k}|^p-a\rho_{p,k}^{\gamma} \partial_x u_{p,k}\right]\,\mathrm{d}x=0.
		\end{equation*}
		From the continuity equation, we deduce
		\begin{equation*}
			-(\gamma-1)\rho_{p,k}^{\gamma} \partial_x u_{p,k}=\left(\rho_{p,k}^{\gamma}\right)_t +\partial_x(\rho_{p,k}^{\gamma}u_{p,k}).
		\end{equation*}
		Thus
		\begin{equation*}
			\frac{\mathrm{d}}{\mathrm{d}t}\int_{\Omega_k} \frac{\rho_{p,k}^{\gamma}}{\gamma-1}\,\mathrm{d}x =-\int_{\Omega_k} \rho_{p,k}^{\gamma} \partial_x u_{p,k} \,\mathrm{d}x.
		\end{equation*}
		Adding the two identities and integrating over $(0,t)$ yields desired estimates \eqref{es: energy pk}. 	
		
		Since $0\leq \phi_k(x)\leq 1$ and $\rho_{\star,k} = \operatorname{ess}\inf_{\Omega_k}\rho_0\leq \rho_0(x)$, we obtain the point-wise bounds
		\begin{equation*}
			\rho_{k,0}(x)=\rho_0(x)\phi_k(x)+\rho_{\star, k}(1-\phi_k(x)) \leq \rho_0(x),
			\qquad |u_{k,0}(x)| = |u_0(x)|\phi_k(x)\leq |u_0(x)|.
		\end{equation*}
		Therefore, for the smooth approximations $(\rho_{p,k,0},u_{p,k,0})$, which preserve these bounds, we have
		\begin{align*}
			\int_{\Omega_k}\left(\frac{1}{2}\rho_{p,k,0}u_{p,k,0}^2+\frac{a}{\gamma-1}\rho_{p,k,0}^\gamma\right)\,\mathrm{d}x \leq&
			\int_{\Omega_k}\left(\frac{1}{2}\rho_{k,0} u_{k,0}^2+\frac{a}{\gamma-1}\rho_{k,0}^\gamma\right)\,\mathrm{d}x \\
			\leq &
			\int_{\mathbb{R}}\left(\frac{1}{2}\rho_0 u_0^2+\frac{a}{\gamma-1}\rho_0^\gamma\right)\,\mathrm{d}x.
		\end{align*}
		This completes the proof.
	\end{proof}
	\subsection{Uniform estimates independent of $p$}
	The following proposition establishes the key pointwise bounds for the density and the stress of the truncated system~\eqref{PDE,pk}, which are uniform in the power-law index $p$. Our proof is a straightforward adaptation of that in \cite{Bresch2026};  only slight modifications are needed to deal with the homogenous Dirichlet boundary data, in contract to the periodic condition in \cite{Bresch2026}. We present a detailed proof in Appendix~\ref{sec: appendix}, without claiming any originality.
	
	\begin{proposition}\label{prop: bound for stress and rho}
		Let $(\rho_{p,k},u_{p,k})$ be a smooth solution of system \eqref{PDE,pk}. Then, we have for $(t,x)\in[0,T]\times \Omega_k$ that
		\begin{equation}\label{es: bound stress}
			\mu\left|\partial_x u_{p,k}\right|^{p-2}\partial_x u_{p,k}-a\rho_{p,k}^{\gamma}\leq \mu\left|\partial_x u_{p,k,0}\right|^{p-2}\partial_x u_{p,k,0}-a\rho_{p,k,0}^{\gamma}\leq \mu,
		\end{equation}
		and
		\begin{equation}\label{es: bound rho_pk}
			\frac{c_1(\Omega_k) e^{-2t}}{\max\left\{1,\left(\frac{a}{\mu}\right)^{\frac{1}{\gamma}}c_2\right\}}\leq \rho_{p,k}(x,t)\leq \sup\limits_{x\in\Omega_k}\rho_{p,k,0}(x)\exp\left[C(T,a,\mu,\gamma,k,E_0)+\frac{3}{2}t\right],
		\end{equation}	
		where $c_1(\Omega_k)$ is the positive constant depending on $\Omega_k$ and $c_2$ depends on the initial data as in Theorem \ref{thm: main}.
	\end{proposition}
	\begin{proof}[Sketch proof of Proposition \ref{prop: bound for stress and rho}]
		Following \cite{Bresch2026}, we set $\sigma_{p,k}=\mu|\partial_x u_{p,k}|^{p-2}\partial_x u_{p,k}-a\rho_{p,k}^\gamma$ and prove that
		\begin{equation*}
			\sigma_{p,k}(t,x)\leq \sup_{x\in\Omega_k}\sigma_{p,k}(0,x)\leq \mu
		\end{equation*}
		for all $(t,x)\in[0,T]\times\overline{\Omega}_k$.
		We first consider two trivial cases. If $\sigma_{p,k}$ is constant in space, then the desired bound follows immediately from the initial condition, since $\sigma_{p,k}(t,x)=\sigma_{p,k}(0,x)\leq \mu$ for all $t\geq 0$. On the other hand, if the maximum value of $\sigma_{p,k}$ is non-positive, then the conclusion $\sigma_{p,k}\leq \mu$ is trivially satisfied because $\mu>0$. Hence, in the sequel, we may assume that $\sigma_{p,k}$ is non-constant and its maximum is strictly positive.
		
		Direct computation (see Appendix~\ref{sec: appendix}) leads to the parabolic equation for $\sigma_{p,k}$:
		\begin{equation}\label{eq: dt sigma}
			\begin{aligned}
				&\partial_{t}\sigma_{p,k} + u_{p,k}\partial_{x}\sigma_{p,k} - \mu H'\left(\partial_{x}u_{p,k}\right)\partial_{x}\left(\frac{1}{\rho_{p,k}}\partial_{x}\sigma_{p,k}\right)\nonumber\\
				&\qquad= -\gamma \sigma_{p,k}\partial_{x}u_{p,k} + \mu (1 + \gamma -p)\left|\partial_{x}u_{p,k}\right|^{p},
			\end{aligned}
		\end{equation}
		where $H(s)=|s|^{p-2}s$.
		
		For an interior maximum point $x_1\in \Omega_k$, the classical parabolic maximum principle applies. Since the right-hand side of \eqref{eq: dt sigma} is non-positive due to the monotonicity of $H^{-1}$, we conclude that $\sigma_{p,k}\leq \mu$ in the interior.
		
		To exclude the boundary maximum points, we apply a new argument that has not appeared in \cite{Bresch2026}. Indeed, if the maximum point $x_{1}$ lies on the boundary, without loss of generality, suppose that $x_{1} = 2k + 2$. From the homogeneous Dirichlet boundary condition, we infer that $\partial_{t}u_{p,k}(t,x_{1}) = 0$ for all $t\geq 0$. However, the momentum equation~\eqref{PDE,pk} reads 
		\begin{equation}\label{eq: momentum boundary}
			-\mu\partial_x\left(|\partial_x u_{p,k}|^{p-2}\partial_x u_{p,k}\right)+a\,\partial_x\rho_{p,k}^\gamma =0\qquad\text{at }x=x_1,
		\end{equation}
		Hence, we have the Neumann condition $\partial_{x}\sigma_{p,k}(t,x_{1}) = 0$. As $\sigma_{p,k}(t,x_{1}) = M > 0$ is a maximum, Hopf's lemma implies that the inward normal derivative must be strictly positive:
		\begin{equation*}
			\lim_{h\rightarrow 0^+} \frac{\sigma_{p,k}(t,x_1)-\sigma_{p,k}(t,x_1-h)}{h}> 0.
		\end{equation*}
	    This leads to contradiction. Hence, a boundary maximum cannot exist, so $\sigma_{p,k}(t,x)\leq \mu$ for every $(t,x)\in[0,T]\times\overline{\Omega}_k$. Consequently,
		\begin{equation*}
			\sup_{x\in\Omega_k} \left|\partial_x u_{p,k}\right|^{p-2}\partial_x u_{p,k}\leq \frac{a}{\mu}\rho_{p,k}^{\gamma}+1,
		\end{equation*}
		and hence
		\begin{equation}\label{es: u_x}
			\partial_x u_{p,k}\leq \left(\frac{a}{\mu}\rho_{p,k}^{\gamma}(t,x)+1\right)^{\frac{1}{p-1}}\quad \text{for any } x\in \Omega_k.
		\end{equation}
		
		With \eqref{es: u_x} at hand, we may follow the arguments in Bresch--Burtea--Szlenk~\cite{Bresch2026} to conclude. We only sketch the proof here and refer to Appendix~\ref{sec: appendix} for details. Indeed, by integrating along the flow map, we arrive at the lower bound for $\rho_{p,k}$: 
		\begin{align*}
			\rho_{p,k}(t,X_{t}(x))&\geq \frac{\rho_{0}e^{-2t}}{\max\left\{1,\left(\frac{a}{\mu}\right)^{1/\gamma}\|\rho_{0}\|_{L^{\infty}}\right\}}\\
			&\geq \frac{c_1(\Omega_k) e^{-2t}}{\max\left\{1,\left(\frac{a}{\mu}\right)^{\frac{1}{\gamma}}c_2\right\}}.
		\end{align*}
		For the upper bound, we use the potential method as in Basov--Shelukhin~\cite{Basov1999}, but with a modified definition of $\psi_{p,k}$ adapted to the Dirichlet boundary condition:
		\begin{align}
			\label{basov-shelukhin}
			\psi_{p,k}(t,x) &= \int_{0}^{t}\left(\rho_{p,k}u_{p,k}^{2} - \mu |\partial_{x}u_{p,k}|^{p - 2}\partial_{x}u_{p,k} + a\rho_{p,k}^{\gamma}\right)(s,x)\,\mathrm{d}s\nonumber\\
			&\qquad 
			- \frac{1}{|\Omega_{k}|}\int_{\Omega_{k}}\left(\int_{y}^{x}\left(\rho_{p,k,0}u_{p,k,0}\right)(z)\,\mathrm{d}z\right)\,\mathrm{d}y.
		\end{align}
		It holds (\emph{cf.} Appendix~\ref{sec: appendix}) that 
		\begin{equation}
			\label{crucial estimate for psi}
			\left|\psi_{p,k}(t,x) - \psi_{p,k}(0,x)\right|\leq C(T,a,\mu,\gamma,k,E_0) + \frac{\mu}{2} t, 
		\end{equation}
		which leads to 
		\begin{equation*}
			\rho_{p,k}(x,t)\leq \sup\limits_{x\in\Omega_k}\rho_{p,k,0}(x)\exp\left[C(T,a,\mu,\gamma,k,E_0)+\frac{3}{2}t\right].
		\end{equation*}
		From here, we may deduce estimates for time derivatives using ideas introduced by Hoff~\cite{Hoff1987,Hoff1995}. The estimates on $\dot{u}_{p,k}$ are analogous to those in \cite[Proposition 2.4]{Bresch2026}. 
	\end{proof}

	The detailed proof of the following two propositions are also deferred to Appendix~\ref{sec: appendix}; only the key ideas are sketched below.
	\begin{proposition}
		\label{prop: rho dotu}
		Let $(\rho_{p,k},u_{p,k})$ be a smooth solution of Eq.~\eqref{PDE,pk}. Then, the following estimates hold
		\begin{equation}
			\int_{0}^{t}\int_{\Omega_k}\rho_{p,k}\left|\dot{u}_{p,k}\right|^2(s,x)\,\mathrm{d}x\,\mathrm{d}s+\frac{\mu}{2}\int_{\Omega_k}\frac{\left|\partial_x u_{p,k}\right|^p}{p}(t,x)\,\mathrm{d}x\leq C(T,\gamma,\mu,a, E_0,k).
		\end{equation}
	\end{proposition}
	\begin{proof}[Sketch proof of Proposition \ref{prop: rho dotu}]
		The proof follows by testing the momentum equation with $\dot{u}_{p,k}$. The main difficulty lies in controlling the integral involving $\sigma_{p,k}$, which in the full expansion (see Appendix A) appears as the term $\mathcal{I}_4$. To handle it, we use the representation
		\begin{equation*}
			\sigma_{p,k}(t,x) = \int_{x(t)}^{x}(\rho_{p,k}\dot{u}_{p,k})(t,z)\,\mathrm{d}z - a\rho_{p,k}^{\gamma}(t,x(t)),
		\end{equation*}
		where $x(t)\in \Omega_k$ is chosen such that $\partial_x u_{p,k}(t,x(t))=0$. The existence of $x(t)$ follows from $\int_{\Omega_k}\partial_x u_{p,k}\,\mathrm{d}x=0$. The bounds from the upper bound of density and the energy estimates then gives the desired control of $\mathcal{I}_4$. We refer to the detailed proof in Appendix A for complete expansion.
	\end{proof}
	
	Finally, following~\cite[Proposition 2.5]{Bresch2026}, we combine the previous two propositions to bound the stress tensor in $L^{2}(0,T;L^{\infty}(\Omega_{k}))$. It shall play a crucial role in the passage of the limit $p \to \infty$.
	
	\begin{proposition}\label{prop: bound tau_pk}
		Let $(\rho_{p,k},u_{p,k})$ be a smooth solution of \eqref{PDE,pk}. Then
		\begin{equation*}
			|\partial_{x}u_{p,k}|^{p-2}\partial_{x}u_{p,k} \in L^2(0,T;L^{\infty}(\Omega_k))\quad \text{uniformly in p}.
		\end{equation*}
	\end{proposition}
	\begin{proof}[Sketch proof of Proposition \ref{prop: bound tau_pk}]
		It follows from the integral representation of $\sigma_{p,k}$ in the sketched proof of Proposition~\ref{prop: rho dotu} and the uniform bounds for $\rho_{p,k}$, $\dot{u}_{p,k}$ in  Propositions~\ref{prop: bound for stress and rho} and \ref{prop: rho dotu}.
	\end{proof}
	
	\section{The limit passage as $p\rightarrow\infty$}\label{sec: pass to lim}
	We now establish the convergence of the approximate solutions $(\rho_{p,k},u_{p,k})$ as $p\rightarrow\infty$ and identify the limit system. Throughout this section, $k$ is a fixed sufficiently large integer (as in Theorem \ref{thm: local}), and write $\Omega_k=(-2k-2,2k+2)$.
	\begin{theorem}
		\label{thm: p_limit}
		Under the assumptions of Theorem \ref{thm: local}, there exists a subsequence $\left\{p_j\right\}\subset\left\{p\right\} \rightarrow \infty$ such that, for every $1\leq r<\infty$, one has that
		\begin{align}
			u_{p_j,k} &\rightharpoonup u_k \quad \text{weakly in } L^2(0,T; H^1(\Omega_k))\cap L^r(0,T;W^{1,r}(\Omega_k)),\nonumber\\
			u_{p_j,k} &\longrightarrow u_k \quad \text{strongly in } L^2((0,T)\times\Omega_k),\\
			\tau_{p_j,k} := |\partial_x u_{p_j,k}|^{p_j-2}\partial_x u_{p_j,k} &\rightharpoonup \tau_k \quad \text{weakly in } L^2((0,T)\times\Omega_k),\nonumber\\
			\rho_{p_j,k} &\longrightarrow \rho_k \quad \text{strongly in } C([0,T]; L^r(\Omega_k)).\nonumber
		\end{align}
		Moreover, the limit $(\rho_{k},u_{k})$ satisfies the following system on $\Omega_{k}$:
		\begin{equation}\label{PDE,k}
			\left\{
			\begin{array}{ll}
				\partial_{t}\rho_{k} + \partial_{x}(\rho_{k}u_{k}) = 0,\\
				\partial_{t}(\rho_{k}u_{k}) + \partial_{x}(\rho_{k}u_{k}^{2}) - \mu \partial_{x}\tau_{k} + a\partial_{x}\rho_{k}^{\gamma} = 0,\\
				\tau_{k} = \pi_{k}\partial_{x}u_{k},\quad \pi_{k}\geq 0,\quad \pi_{k}(1 - |\partial_{x}u_{k}|) = 0,\text{ and }\quad |\partial_{x}u_{k}|\leq 1 \quad \text{a.e. in } (0,T)\times \Omega_{k}.
			\end{array}
			\right.
		\end{equation}
	\end{theorem}
	\begin{proof}
		We divide our proof in the seven steps below.
		
		\smallskip
		\noindent
		{\bf Step~1: Uniform estimates for $u_{p,k}$.}
	From the energy estimates in Proposition \ref{prop: energy}, there exists a constant $C(E_0)>0$ independent of $p$ such that
	\begin{equation*}
		\int_0^T \int_{\Omega_k} |\partial_x u_{p,k}|^p \,\mathrm{d}x \,\mathrm{d}t\leq C(E_0).
	\end{equation*}
	For any fixed $r<\infty$ and $p\geq r$, H\"{o}lder's inequality gives us
	\begin{align*}
		\left(\int_0^T \left\|\partial_x u_{p,k}\right\|_{L_x^r}^r \,\mathrm{d}t\right)^{\frac{1}{r}} \leq& \left(|\Omega_k|\times T\right)^{\frac{1}{r}-\frac{1}{p}}\left(\int_0^T \left\|\partial_x u_{p,k}\right\|_{L_x^p}^p\right)^{\frac{1}{p}}\\
		\leq & \left(|\Omega_k|\times T\right)^{\frac{1}{r}}\left(\frac{C(E_0)}{|\Omega_k|\times |T|}\right)^{\frac{1}{p}}\\
		\leq &\max\left\{ \left(|\Omega_k|\times T\right)^{\frac{1}{r}},C(E_0)^{\frac{1}{r}}\right\}.
	\end{align*}
	Thus, there exists a subsequence, still denoted by $u_{p,k}$, such that
	\begin{equation}\label{wlimit ux}
		\partial_{x}u_{p,k}\rightharpoonup \partial_{x}u_k\quad\text{in } L^r((0,T)\times \Omega_k).
	\end{equation}
	On the other hand, for any $p\geq 2$, we have
	\begin{align}\label{eq: u_x L^2}
		\int_0^T\int_{\Omega_k}\left|\partial_x u_{p,k}\right|^2\,\mathrm{d}x\,\mathrm{d}t\leq& \int_0^T \int_{\Omega_k}\left(1+\left|\partial_x u_{p,k}\right|^p\right) \mathrm{d}x\,\mathrm{d}t \nonumber\\
		\leq &C(T,E_0,k).
	\end{align}
	By \eqref{es: bound rho_pk}, $\rho_{p,k}$ is bounded and has a positive lower bound depending on $k$. Thus, from the bounds
	$\|\sqrt{\rho_{p,k}} u_{p,k}\|_{L^\infty(0,T;L^2(\Omega_k))}\leq E_0$ and $\|\sqrt{\rho_{p,k}}\dot u_{p,k}\|_{L^\infty(0,T;L^2(\Omega_k))}\leq C$ in Propositions \ref{prop: energy} and \ref{prop: rho dotu}, we infer
	\begin{equation}\label{es: u dotu}
		\|u_{p,k}\|_{L^\infty(0,T;L^2(\Omega_k))}+\|\dot u_{p,k}\|_{L^\infty(0,T;L^2(\Omega_k))}\leq C(E_0,k).
	\end{equation}
	Next we estimate $u_{p,k}\partial_x u_{p,k}$ in $L^2(0,T;L^2(\Omega_k))$. Using the one-dimensional Sobolev embedding $H^1(\Omega_k)\hookrightarrow L^\infty(\Omega_k)$,
	\begin{align}\label{eq: uu_x}
		\int_0^T\left\|u_{p,k}\partial_x u_{p,k}\right\|_{L_x^2}^2\,\mathrm{d}t= &\int_0^T \int_{\Omega_k} |u_{p,k}|^2 |\partial_x u_{p,k}|^2\,\mathrm{d}x \,\mathrm{d}t\nonumber\\
		\leq & \int_0^T\left\| u_{p,k}\right\|_{L_x^{\infty}}^2 \left\|\partial_x u_{p,k}\right\|_{L_x^{2}}^2\,\mathrm{d}t \nonumber\\
		\leq & \int_0^T \left\|u_{p,k}\right\|_{H_x^1}^2\left\|\partial_x u_{p,k}\right\|_{L_x^{2}}^2\,\mathrm{d}t\\
		\leq &\int_0^T\left(\left\|u_{p,k}\right\|_{L_x^2}^2+\left\|\partial_x u_{p,k}\right\|_{L_x^2}^2\right)\left\|\partial_x u_{p,k}\right\|_{L_x^{2}}^2\,\mathrm{d}t.\nonumber
	\end{align}
	From \eqref{es: u dotu}, $\|u_{p,k}\|_{L^2}^2$ is bounded uniformly in $t$, so the first part of \eqref{eq: uu_x} is controlled by
	\begin{align*}
		\int_0^T \left\|u_{p,k}\right\|_{L_x^2}^2 \left\|\partial_x u_{p,k}\right\|_{L_x^{2}}^2\,\mathrm{d}t\leq& \sup_{t\in(0,T)}\left\|u_{p,k}\right\|_{L_x^2}^2\int_0^T\left\|\partial_x u_{p,k}\right\|_{L_x^{2}}^2\,\mathrm{d}t\\
		\leq &C(T,E_0,k).
	\end{align*}
	For the second term, for sufficiently large $p$ with $1-\frac{4}{p}\geq 0$, H\"{o}lder inequality gives
	\begin{align*}
		\int_0^T \left\|\partial_x u_{p,k}\right\|_{L_x^{2}}^4\,\mathrm{d}t \leq& \int_0^T \left|\Omega_k\right|^{2-\frac{2}{p}}\left\|\partial_x u_{p,k}\right\|_{L_x^p}^4\,\mathrm{d}t\\
		\leq &\left|\Omega_k\right|^{2-\frac{2}{p}}T^{1-\frac{4}{p}}\left(\int_0^T \left\|\partial_x u_{p,k}\right\|_{L_x^p}^p\,\mathrm{d}t\right)^{\frac{4}{p}}\\
		\leq &|\Omega_k|^{2-\frac{2}{p}}T^{1-\frac{4}{p}}\left(\int_0^T \left\|\partial_x u_{p,k}\right\|_{L_x^p}^p\,\mathrm{d}t\right)^{\frac{4}{p}}\\
		\leq &C(T,E_0,k).
	\end{align*}
	Therefore, $u_{p,k}\partial_x u_{p,k}\in L^2(0,T;L^2(\Omega_k))$ uniformly in $p$.
	
	Finally, from $\partial_t u_{p,k}=\dot{u}_{p,k}-u_{p,k}\partial_x u_{p,k}$, we conclude that
	\begin{equation*}
		\partial_t u_{p,k}\in L^2(0,T;L^2(\Omega_k))\subset L^2(0,T;H^{-1}(\Omega_k)).
	\end{equation*}
	We have also shown that
	\begin{equation*}
		u_{p,k}\in L^2(0,T;H^1(\Omega_k)).
	\end{equation*}
    Thus, in view of the compact embedding $H^{1}(\Omega_{k})\hookrightarrow\hookrightarrow L^{2}(\Omega_{k})$, the compactness lemma of Aubin--Lions--Simon (Theorem \ref{thm:ALS}) yields that $\{u_{p,k}\}_{p \in \mathbb{N}}$ is relatively compact in $L^{2}(0,T;L^{2}(\Omega_{k}))$. Hence, up to a subsequence (unrelabelled),
	\begin{equation*}
		u_{p,k}\longrightarrow u_{k} \quad\text{strongly in}\quad  L^2(0,T;L^2(\Omega_k)).
	\end{equation*}   
	
	\smallskip
	\noindent
	{\bf Step~2: Limit for $\partial_{x}u_{k}$.} We again argue as in \cite{Bresch2026}. For any $\eta >0$, define
	\begin{equation*}
		E_{\eta}:=\left\{(t,x):\left|\partial_x u_{k}\right|\geq 1+\eta\right\}.
	\end{equation*} 
	By Chebyshev's inequality and H\"{o}lder's inequality
	\begin{align*}
		&(1+\eta)|E_{\eta}|\leq \int_{0}^{T}\int_{\Omega_k}\left|\partial_x u_{k}\right|\mathbb{I}_{E_{\eta}}\,\mathrm{d}x\,\mathrm{d}t\\
		\leq &\liminf_{p\rightarrow \infty}\int_{0}^{T}\int_{\Omega_k}\left|\partial_x u_{p,k}\right|\mathbb{I}_{E_{\eta}}\,\mathrm{d}x\,\mathrm{d}t\\
		\leq &\liminf_{p\rightarrow \infty}\left(\int_{0}^{T}\int_{\Omega_k}\left|\partial_x u_{p,k}\right|^p \,\mathrm{d}x \,\mathrm{d}t\right)^{\frac{1}{p}}|E_{\eta}|^{1-\frac{1}{p}}\\
		\leq & \liminf_{p\rightarrow \infty}C(E_{0})^{\frac{1}{p}}|E_{\eta}|^{1-\frac{1}{p}},
	\end{align*}
	where in the last line we used the uniform bound $\int_0^T\int_{\Omega_k}|\partial_x u_{p,k}|^p \,\mathrm{d}x \,\mathrm{d}t\leq C(E_0)$ from Proposition \ref{prop: energy}. Since $C(E_0)^{\frac{1}{p}}\rightarrow 1$ and $|E_{\eta}|^{1-\frac{1}{p}}\rightarrow |E_{\eta}|$ as $p\rightarrow\infty$, we obtain
	\begin{equation*}
		\left(1+\eta\right)|E_{\eta}|\leq \liminf_{p\rightarrow \infty}|E_{\eta}|^{1-\frac{1}{p}}=|E_{\eta}|,
	\end{equation*}
	which is impossible unless $|E_{\eta}|=0$ for every $\eta>0$. Thus, $|\partial_x u_k|\leq 1$ \emph{a.e.} in $(0,T)\times \Omega_k$.
	
	\smallskip
	\noindent
	{\bf Step~3: Compactness of $\rho_{p,k}$.}
	The continuity equation gives us  $\partial_{t}\rho_{p,k} = -\partial_{x}(\rho_{p,k}u_{p,k})$. Using the uniform boundedness of $\rho_{p,k}$ and the energy estimates, we have
	\begin{align*}
		|\langle\partial_t \rho_{p,k},\varphi\rangle| 
		&= \left|\int_{\Omega_k} \rho_{p,k}u_{p,k} \partial_x\varphi\,\mathrm{d}x\right| \\
		&\leq \|\rho_{p,k}\|_{L_x^\infty(\Omega_k)}^{1/2} \|\rho_{p,k}^{1/2}u_{p,k}\|_{L_x^2(\Omega_k)} \|\partial_x\varphi\|_{L_x^2(\Omega_k)} \\
		&\leq C(E_0) \|\varphi\|_{H^1},
	\end{align*}
	for any test function $\phi\in H_0^1(\Omega_k)$.
	Thus, $\partial_t \rho_{p,k}\in L^\infty(0,T;H^{-1}(\Omega_k))$ uniformly in $p$. 
	Moreover, $\rho_{p,k}$ is uniformly bounded in $L^\infty(0,T;L^\infty(\Omega_k))$, hence in $L^\infty(0,T;L^r(\Omega_k))$ for any $r<\infty$.
	
	Since the embedding $L^r(\Omega_k)\hookrightarrow\hookrightarrow H^{-1}(\Omega_k)$ is compact, by Theorem~\ref{thm:ALS}
	with $X=L^r(\Omega_k)$, $Y=H^{-1}(\Omega_k)$ and $p=\infty$, we obtain a subsequence (still denoted $\rho_{p,k}$) such that
	\begin{equation*}
		\rho_{p,k}\longrightarrow\rho_k \quad\text{strongly in } C([0,T];H^{-1}(\Omega_k)).
	\end{equation*}
	On the other hand, recall the momentum equation
	\begin{equation*}
		\partial_t(\rho_{p,k}u_{p,k}) + \partial_x(\rho_{p,k}u_{p,k}^2) - \mu\partial_x(|\partial_x u_{p,k}|^{p-2}\partial_x u_{p,k}) + a\partial_x\rho_{p,k}^\gamma = 0.
	\end{equation*}
	Testing this equation with $u_{p,k}$ and with $u_k$, we deduce that
	\begin{align*}
		\frac{a}{\gamma-1}\int_{\Omega_k}\rho_{p,k}^\gamma(t,x)\,\mathrm{d}x
		&+\int_0^t\int_{\Omega_k}\left(\rho_{p,k}\dot{u}_{p,k}\right)u_{p,k}\,\mathrm{d}x \,\mathrm{d}s
		+\int_0^t\int_{\Omega_k}|\partial_x u_{p,k}|^p \,\mathrm{d}x \,\mathrm{d}s\\
		&=\frac{a}{\gamma-1}\int_{\Omega_k}\rho_{p,k,0}^\gamma(x)\,\mathrm{d}x,
	\end{align*}
	and
	\begin{align*}
		\frac{a}{\gamma-1}\int_{\Omega_k}\rho_k^\gamma(t,x)\,\mathrm{d}x
		&-a\int_0^t\int_{\Omega_k}\left(\rho_{p,k}^\gamma-\rho_k^\gamma\right)\partial_x u_k \,\mathrm{d}x \,\mathrm{d}s
		+\int_0^t\int_{\Omega_k}\left(\rho_{p,k}\dot{u}_{p,k}\right)u_k \,\mathrm{d}x \,\mathrm{d}s\\
		&+\int_0^t\int_{\Omega_k}|\partial_x u_{p,k}|^{p-2}\partial_x u_{p,k}\,\partial_x u_k \,\mathrm{d}x \,\mathrm{d}s
		=\frac{a}{\gamma-1}\int_{\Omega_k}\rho_{k,0}^\gamma(x) \,\mathrm{d}x.
	\end{align*}
	Subtracting the two identities yields that
	\begin{align*}
		&\frac{a}{\gamma-1}\int_{\Omega_k}\left(\rho_{p,k}^\gamma-\rho_k^\gamma\right)(t)\,\mathrm{d}x
		+ a\int_0^t\int_{\Omega_k}\left(\rho_{p,k}^\gamma-\rho_k^\gamma \right)\partial_x u_k \,\mathrm{d}x\,\mathrm{d}s
		\\&+ \int_0^t\int_{\Omega_k}\rho_{p,k}\dot{u}_{p,k}(u_{p,k}-u_k)\,\mathrm{d}x\,\mathrm{d}s
		+ \int_0^t\int_{\Omega_k}|\partial_x u_{p,k}|^{p-2}\partial_x u_{p,k}\left(\partial_x u_{p,k}-\partial_x u_k\right)\,\mathrm{d}x\,\mathrm{d}s\\ 
		&=\frac{a}{\gamma-1}\int_{\Omega_k}\left(\rho_{p,k,0}^\gamma-\rho_{k,0}^\gamma\right)(x)\,\mathrm{d}x.
	\end{align*}
	To proceed, by monotonicity of the function $s\mapsto |s|^{p-2}s$, one has that
	\begin{equation*}
		\left(|\partial_x u_{p,k}|^{p-2}\partial_x u_{p,k}-|\partial_x u_k|^{p-2}\partial_x u_k\right)\left(\partial_x u_{p,k}-\partial_x u_k\right)\geq 0.
	\end{equation*}
	It follows that
	\begin{equation*}
		|\partial_x u_{p,k}|^{p-2}\partial_x u_{p,k}\left(\partial_x u_{p,k}-\partial_x u_k\right)
		\geq |\partial_x u_k|^{p-2}\partial_x u_k\left(\partial_x u_{p,k}-\partial_x u_k\right),
	\end{equation*}
	and hence
	\begin{align*}
		&\frac{a}{\gamma-1}\int_{\Omega_k}\left(\rho_{p,k}^\gamma-\rho_k^\gamma\right)(t)\,\mathrm{d}x
		+ a\int_0^t\int_{\Omega_k}\left(\rho_{p,k}^\gamma-\rho_k^\gamma \right)\partial_x u_k \,\mathrm{d}x\,\mathrm{d}s
		\\&+ \int_0^t\int_{\Omega_k}\rho_{p,k}\dot{u}_{p,k}(u_{p,k}-u_k)\,\mathrm{d}x\,\mathrm{d}s
		+ \int_0^t\int_{\Omega_k}|\partial_x u_k|^{p-2}\partial_x u_k\left(\partial_x u_{p,k}-\partial_x u_k\right)\,\mathrm{d}x\,\mathrm{d}s\\ 
		&\leq \frac{a}{\gamma-1}\int_{\Omega_k}\left(\rho_{p,k,0}^\gamma-\rho_{k,0}^\gamma\right)(x)\,\mathrm{d}x.
	\end{align*}
	Recall from \eqref{wlimit ux} that $\partial_x u_{p,k}\rightharpoonup\partial_x u_k$ in $L^2(0,T;L^2(\Omega_k))$ and $|\partial_x u_k|\leq 1$. It then holds that
	\begin{align*}
		\int_0^t\int_{\Omega_k}|\partial_x u_k|^{p-2}\partial_x u_k\left(\partial_x u_{p,k}-\partial_x u_k\right)\,\mathrm{d}x \,\mathrm{d}s\longrightarrow 0.
	\end{align*}
	Thus, by the strong convergence $u_{p,k}\to u_{k}$ in $L^{2}(0,T;L^{2}(\Omega_{k}))$ and the boundedness of  $\sqrt{\rho_{p,k}}|\dot{u}_{p,k}|$ in $L^{2}((0,T)\times \Omega_{k})$ established in Proposition \ref{prop: rho dotu}, we have that
	\begin{align*}
		&\int_0^t\int_{\Omega_k}\rho_{p,k}\dot{u}_{p,k}(u_{p,k}-u_k)\,\mathrm{d}x\,\mathrm{d}s\\
		\leq& \|\rho_{p,k}\|_{L^{\infty}}^{\frac{1}{2}}\|\sqrt{\rho_{p,k}}|\dot{u}_{p,k}|\|_{L^2((0,T)\times \Omega_k)}\|u_{p,k}-u_k\|_{L^2((0,T)\times \Omega_k)}\longrightarrow 0.
	\end{align*}
	Moreover, the strong convergence $\rho_{p,k,0}\rightarrow \rho_{k,0}$ implies that
	\begin{equation*}
		\int_{\Omega_k}\left(\rho_{p,k,0}^\gamma-\rho_{k,0}^\gamma\right)(x)\,\mathrm{d}x\rightarrow 0.
	\end{equation*}
	Therefore, for \emph{a.e.} $t\in (0,T)$,
	\begin{equation}
		\label{xx}
		\frac{a}{\gamma-1}\int_{\Omega_k}\left(\rho_{p,k}^\gamma(t)-\rho_k^\gamma(t)\right)\,\mathrm{d}x
		+a\int_0^t\int_{\Omega_k}\left(\rho_{p,k}^\gamma-\rho_k^\gamma\right)\partial_x u_k\,\mathrm{d}x\,\mathrm{d}s \leq \varepsilon_p(t),
	\end{equation}
	for some $\varepsilon_p(t)\rightarrow 0$ as $p\rightarrow \infty$.
	
	In addition, let us define
	\begin{equation*}
		X_{p}(t)=\int_{\Omega_k}\left(\rho_{p,k}^\gamma-\rho_k^\gamma-\gamma \rho_k^{\gamma-1}(\rho_{p,k}-\rho_k)\right)(t,x)\,\mathrm{d}x.
	\end{equation*}
	Then we have the identities
	\begin{equation*}
		\int_{\Omega_k}\left(\rho_{p,k}^\gamma(t)-\rho_k^\gamma(t)\right)\,\mathrm{d}x=X_p(t)+\gamma\int_{\Omega_k}\left( \rho_k^{\gamma-1}(\rho_{p,k}-\rho_k)\right)(t,x)\,\mathrm{d}x
	\end{equation*}
	and
	\begin{align*}
		\int_0^t\int_{\Omega_k}\left(\rho_{p,k}^\gamma-\rho_k^\gamma \right)\partial_x u_k \,\mathrm{d}x\,\mathrm{d}s=&\int_0^t\int_{\Omega_k}\left(\rho_{p,k}^\gamma-\rho_k^\gamma-\gamma \rho_k^{\gamma-1}(\rho_{p,k}-\rho_k) \right)\partial_x u_k \,\mathrm{d}x\,\mathrm{d}s\\
		&+\int_0^t\int_{\Omega_k}\left(\gamma \rho_k^{\gamma-1}(\rho_{p,k}-\rho_k) \right)\partial_x u_k \,\mathrm{d}x\,\mathrm{d}s.
	\end{align*}
	Together with the pointwise bound $|\partial_x u_k|\leq 1$ and estimate \eqref{xx} just obtained, we deduce that
	\begin{align*}
		&\frac{a}{\gamma-1}X_p(t)+\frac{a\gamma}{\gamma-1}\int_{\Omega_k}\left( \rho_k^{\gamma-1}(\rho_{p,k}-\rho_k)\right)(t,x)\,\mathrm{d}x\\
		+&a\int_0^t\int_{\Omega_k}\left(\rho_{p,k}^\gamma-\rho_k^\gamma-\gamma \rho_k^{\gamma-1}(\rho_{p,k}-\rho_k) \right)\partial_x u_k \,\mathrm{d}x\,\mathrm{d}s\\
		+&a\int_0^t\int_{\Omega_k}\left(\gamma \rho_k^{\gamma-1}(\rho_{p,k}-\rho_k) \right)\partial_x u_k \,\mathrm{d}x\,\mathrm{d}s\leq \varepsilon_p(t).
	\end{align*}
	The strong convergence in $C([0,T];H^{-1}(\Omega_k))$ implies
	\begin{equation*}
		\rho_{p,k}(t)\longrightarrow\rho_k(t)\quad\text{in } H^{-1}(\Omega_k)\ \text{for every }t\in[0,T].
	\end{equation*}

	Since $\rho_{p,k}$ is uniformly bounded in $L^\gamma(\Omega_k)$, for each $t$ there exists a subsequence such that
	\begin{equation*}
	    \rho_{p,k}(t) \rightharpoonup \eta(t) \quad \text{weakly in } L^\gamma(\Omega_k),
	\end{equation*}
	for some $\eta(t)\in L^\gamma(\Omega_k)$. Since the embedding $L^\gamma(\Omega_k)\hookrightarrow H^{-1}(\Omega_k)$ is continuous, the weak convergence in $L^\gamma$ implies weak convergence in $H^{-1}$ to the same limit $\eta(t)$. Since strong convergence in $H^{-1}$ implies weak convergence in $H^{-1}$ to $\rho_k(t)$, by uniqueness of the weak limit in $H^{-1}$, we must have $\eta(t)=\rho_k(t)$. Hence, for every $t\in[0,T]$,
	\begin{equation*}
		\rho_{p,k}(t)\rightharpoonup\rho_k(t)\quad \text{weakly in}\quad  L^\gamma(\Omega_k).
	\end{equation*}
	As $\rho_k^{\gamma-1}\in L^{\gamma^{\prime}}(\Omega_k)$, where $\gamma^{\prime}$ is the conjugate index of $\gamma$, the weak convergence of $\rho_{p,k}$ shows that
	\begin{equation*}
		\int_{\Omega_k}\rho_k^{\gamma-1}(\rho_{p,k}-\rho_k)\,\mathrm{d}x \longrightarrow 0\quad\text{as} \quad p\rightarrow \infty.
	\end{equation*}
	Similarly, since $|\partial_x u_k|\leq 1$, the same weak convergence yields
	\begin{equation*}
		\int_{\Omega_k}\rho_k^{\gamma-1}(\rho_{p,k}-\rho_k)\partial_x u_k \,\mathrm{d}x \longrightarrow 0\quad\text{as} \quad p\rightarrow \infty.
	\end{equation*}
	Therefore, $X_p(t)$ satisfies the simplified inequality
	\begin{equation*}
		X_p(t) \leq \int_0^t X_p(s)\,\mathrm{d}s+\varepsilon_p(t).
	\end{equation*}
	Gr\"{o}nwall's lemma then yields $X_p(t)\rightarrow 0$ for almost every $t$.
	Consequently,
	\begin{equation*}
		\int_{\Omega_k}\rho_{p,k}^\gamma\,\mathrm{d}x \longrightarrow \int_{\Omega_k}\rho_k^\gamma\,\mathrm{d}x \qquad\text{for a.e. }t\in[0,T].
	\end{equation*}
	This together with the weak convergence gives
	\begin{equation*}
		\rho_{p,k}(t)\longrightarrow\rho_k(t)\quad\text{strongly in } L^\gamma(\Omega_k)\ \text{for almost every }t\in[0,T].
	\end{equation*}
	Then we extend the convergence to any $1\leq r<\infty$. Indeed, for \emph{a.e.} $t$, if $r\geq\gamma$, interpolation inequality gives
	\begin{equation*}
		\|\rho_{p,k}(t)-\rho_k(t)\|_{L_x^r(\Omega_k)}\leq \|\rho_{p,k}(t)-\rho_k(t)\|_{L_x^{\gamma}(\Omega_k)}^{\gamma/r}\|\rho_{p,k}(t)-\rho_k(t)\|_{L_x^{\infty}(\Omega_k)}^{1-\gamma/r}\longrightarrow 0,\quad \text{as }p\rightarrow \infty.
	\end{equation*}
	On the other hand, for $r<\gamma$, H\"{o}lder's inequality yields
	\begin{equation*}
		\|\rho_{p,k}(t)-\rho_k(t)\|_{L_x^r(\Omega_k)}\leq |\Omega_k|^{1/r-1/\gamma}\|\rho_{p,k}(t)-\rho_k(t)\|_{L_x^{\gamma}(\Omega_k)}\rightarrow 0,\quad \text{as }p\rightarrow \infty.
	\end{equation*}
	By the Lebesgue dominated convergence theorem, 
	\begin{equation*}
		\int_0^T \|\rho_{p,k}(t)-\rho_k(t)\|_{L^r}^r\,\mathrm{d}t \rightarrow 0,\quad \text{as }p\rightarrow \infty.
	\end{equation*}
	Hence, for any $1\leq r<\infty$,
	\begin{equation*}
		\rho_{p,k}\longrightarrow \rho_k \quad \text{strongly in}\quad  L^r((0,T)\times \Omega_k).
	\end{equation*}
	
	\smallskip
	\noindent
	{\bf Step~4: Passage to the limit.}
	Now, define 
	\begin{equation*}
		\tau_{p,k}:=|\partial_x u_{p,k}|^{p-2}\partial_x u_{p,k}.
	\end{equation*}
	By Proposition \ref{prop: bound tau_pk}, this sequence is bounded in $L^2(0,T;L^{\infty}(\Omega_k))$ uniformly in $p$. Thus,
	\begin{align*}
		\left\|\tau_{p,k}\right\|_{L^2((0,T)\times \Omega_k)}\leq & \left(\int_0^T\left\|\tau_{p,k}\right\|_{L_x^{\infty}(\Omega_k)}^2\left|\Omega_k\right|\,\mathrm{d}t\right)^{\frac{1}{2}}\\
		\leq &\left|\Omega_k\right|^{\frac{1}{2}}\left\|\tau_{p,k}\right\|_{L^2(0,T;L^\infty(\Omega_k))}
		\leq C(T,\mu,E_0,k)|\Omega_k|^{1/2},
	\end{align*}
	so that $\left\{\tau_{p,k}\right\}$ is uniformly bounded in $L^2((0,T)\times \Omega_k)$. We can extract a subsequence (not relabelled) and a limit $\tau_k\in L^2((0,T)\times \Omega_k)$ such that
	\begin{equation*}
		\tau_{p,k} \rightharpoonup \tau_k \quad\text{weakly in } L^2((0,T)\times \Omega_k).
	\end{equation*}
	In particular, $		\tau_{p,k} \rightharpoonup \tau_k$ weakly in $L^1((0,T)\times \Omega_k)$.

	\smallskip
	\noindent
	{\bf Step~5: Identify the limiting PDE.} We now show that the limiting functions $(\rho_k,u_k,\tau_k)$ satisfy the continuity and momentum equations in the distributional sense.
	
	For each $p$, the approximate solution satisfies the continuity equation; \emph{i.e.}, for any test function $\phi \in C_c^\infty((0,T)\times \Omega_k)$, one has 
	\begin{equation}\label{eq: test continuity}
		\int_0^T\int_{\Omega_k}\rho_{p,k}\,\partial_t\varphi \,\mathrm{d}x\,\mathrm{d}t+\int_0^T\int_{\Omega_k} \rho_{p,k}u_{p,k}\,\partial_x\varphi \,\mathrm{d}x\,\mathrm{d}t=0.
	\end{equation}
	From the strong convergences $u_{p,k}\rightarrow u_k$ in $L^2((0,T)\times\Omega_k)$ and $\rho_{p,k}\rightarrow\rho_k$ in $L^r((0,T)\times\Omega_k)$ for any $1\leq r<\infty$, together with the uniform boundedness of $\rho_{p,k}$ \eqref{es: bound rho_pk}, one infers that
	\begin{equation*}
		\rho_{p,k}u_{p,k} \longrightarrow \rho_k u_k \quad \text{strongly in } L^1((0,T)\times\Omega_k).
	\end{equation*}
	Thus, passing to the limit $p\rightarrow\infty$ in \eqref{eq: test continuity}, we arrive at
	\begin{equation*}
		\partial_t\rho_k+\partial_x(\rho_k u_k)=0 \quad \text{in } \mathcal{D}^{\prime}((0,T)\times\Omega_k). 
	\end{equation*}
	Next, the approximate momentum equation reads
	\begin{align*}
		&\int_0^T\int_{\Omega_k} \rho_{p,k}u_{p,k}\,\partial_t\psi \,\mathrm{d}x\,\mathrm{d}t
		+\int_0^T\int_{\Omega_k} \rho_{p,k}u_{p,k}^2\,\partial_x\psi\,\mathrm{d}x\,\mathrm{d}t\\
		&+\mu\int_0^T\int_{\Omega_k} \tau_{p,k}\,\partial_x\psi \,\mathrm{d}x\,\mathrm{d}t
		+a\int_0^T\int_{\Omega_k}\rho_{p,k}^\gamma\,\partial_x\psi \,\mathrm{d}x\,\mathrm{d}t=0,
	\end{align*}
	for any test function $\psi \in C_c^\infty((0,T)\times\Omega_k)$,
	We examine each term as $p\rightarrow\infty$.
	\begin{itemize}
		\item  For the first term, since $\rho_{p,k}u_{p,k} \rightarrow \rho_k u_k$ strongly in $L^1((0,T)\times \Omega_k)$, we have
		\begin{equation*}
			\lim_{p\rightarrow\infty}\int_0^T\int_{\Omega_k} \rho_{p,k}u_{p,k}\,\partial_t\psi \,\mathrm{d}x\,\mathrm{d}t=\int_0^T\int_{\Omega_k} \rho_ku_k\,\partial_t\psi\,\mathrm{d}x\,\mathrm{d}t.
		\end{equation*}
		\item For the second term, the strong convergence of $u_{p,k}$ in $L^2((0,T)\times \Omega_k)$ and the uniform bound of $\rho_{p,k}$ give $\rho_{p,k}u_{p,k}^2 \rightarrow \rho_k u_k^2$ strongly in $L^1((0,T)\times \Omega_k)$. Thus
		\begin{equation*}
			\lim_{p\rightarrow\infty}\int_0^T\int_{\Omega_k} \rho_{p,k}u_{p,k}^2\,\partial_x\psi\,\mathrm{d}x\,\mathrm{d}t=\int_0^T\int_{\Omega_k} \rho_ku_k^2\,\partial_x\psi\,\mathrm{d}x\,\mathrm{d}t.
		\end{equation*}
		\item For the third term, since $\tau_{p,k} \rightharpoonup \tau_k$ weakly in $ L^2((0,T)\times \Omega_k)$ and $\partial_x \psi$ is smooth with compact support, we get
		\begin{equation*}
			\lim_{p\rightarrow\infty}\mu\int_0^T\int_{\Omega_k} \tau_{p,k}\,\partial_x\psi\,\mathrm{d}x\,\mathrm{d}t=\mu\int_0^T\int_{\Omega_k} \tau_k\,\partial_x\psi\,\mathrm{d}x\,\mathrm{d}t.
		\end{equation*}
		\item Finally,
		for the fourth term, the strong convergence $\rho_{p,k}^\gamma \rightarrow \rho_k^\gamma$  in $L^1((0,T)\times \Omega_k)$ gives us
		\begin{equation*}
			\lim_{p\rightarrow\infty}a\int_0^T\int_{\Omega_k} \rho_{p,k}^\gamma\,\partial_x\psi\,\mathrm{d}x\,\mathrm{d}t=a\int_0^T\int_{\Omega_k} \rho_k^\gamma\,\partial_x\psi\,\mathrm{d}x\,\mathrm{d}t.
		\end{equation*}
	\end{itemize}
    Therefore, we arrive at the limiting equation
	\begin{equation*}
		\partial_t(\rho_k u_k)+\partial_x(\rho_k u_k^2)-\mu\partial_x\tau_k + a\partial_x\rho_k^\gamma=0 \quad \text{in } \mathcal{D}^{\prime}((0,T)\times\Omega_k),
	\end{equation*}
	without imposing further relation between $\tau_k$ and $\partial_x u_k$. The identification of $\tau_k$ with the Lagrange multiplier $\pi_k$ will be carried out in the next step, using an energy comparison argument.
	
	\smallskip
	\noindent
	{\bf Step~6: Identification of the limit of the $p$-Laplacian term.}
	The limiting system is proved to satisfy the momentum equation, so we can test it with $u_{k}$ to obtain the energy equality:
	\begin{equation}\label{es: energy k}
		\int_{\Omega_k}\left(\frac{1}{2}\rho_k u_k^2+\frac{a}{\gamma-1}\rho_k^\gamma\right)(t)\,\mathrm{d}x + \mu\int_0^t\int_{\Omega_k}\tau_k\partial_x u_k\,\mathrm{d}x\,\mathrm{d}s = \int_{\Omega_k}\left(\frac{1}{2}\rho_{k,0} u_{k,0}^2+\frac{a}{\gamma-1}\rho_{k,0}^\gamma\right)\,\mathrm{d}x.
	\end{equation}
	Since $\tau_{p,k} \rightharpoonup \tau_k$ in $L^1((0,T)\times \Omega_k)$, by the weak lower semi-continuity of the $L^1$ norm and Young's inequality, we have
	\begin{align*}
		\int_0^T \int_{\Omega_k}\left|\tau_k\right| \mathrm{d} x \mathrm{~d} t & \leq \liminf_{p\rightarrow\infty}\int_0^T \int_{\Omega_k}\left|\tau_{p,k}\right|\,\mathrm{d}x\,\mathrm{d}t\\
		& \leq \liminf_{p\rightarrow\infty} \int_0^T \int_{\Omega_k}\left|\partial_xu_{p,k}\right|^{p-1}\,\mathrm{d}x\,\mathrm{d}t\\
		& \leq \liminf_{p\rightarrow\infty}\left(\frac{p-1}{p}\int_0^T \int_{\Omega_k}\left|\partial_x u_{p,k}\right|^p \,\mathrm{d}x\,\mathrm{d}t+\frac{1}{p} T\times|\Omega_k| \right)\\
		&\leq \liminf_{p\rightarrow \infty}\int_0^T \int_{\Omega_k}\left|\partial_x u_{p,k}\right|^p \,\mathrm{d}x\,\mathrm{d}t+C(T,k).
	\end{align*}
	Using the strong convergence $u_{p,k}\rightarrow u_k$ in $L^2((0,T)\times \Omega_k)$ and $\rho_{p,k}\rightarrow\rho_k$ in $L^r((0,T)\times \Omega_k)$, we pass to the limit $p\rightarrow \infty$ in \eqref{es: energy pk}
	\begin{equation*}
		\int_{\Omega_k}\left(\frac{1}{2}\rho_k u_k^2+\frac{a}{\gamma-1}\rho_k^\gamma\right)(t)\,\mathrm{d}x + \mu\int_0^t\int_{\Omega_k} |\tau_k|\,\mathrm{d}x\,\mathrm{d}s \leq \int_{\Omega_k}\left(\frac{1}{2}\rho_{k,0} u_{k,0}^2+\frac{a}{\gamma-1}\rho_{k,0}^\gamma\right)\,\mathrm{d}x.
	\end{equation*}
	Therefore, we have
	\begin{equation*}
		\int_0^t\int_{\Omega_k}|\tau_k|\,\mathrm{d}x\,\mathrm{d}s\leq \int_0^t\int_{\Omega_k} \tau_k\partial_x u_k\,\mathrm{d}x\,\mathrm{d}s.
	\end{equation*}
    By $|\partial_x u_k|\leq 1$, we must have $\tau_k\partial_x u_k=|\tau_k|$ \emph{a.e.}, which implies that
	\begin{equation*}
		|\tau_k|\,(1-|\partial_x u_k|)=0,\qquad \tau_k = |\tau_k|\,\partial_x u_k,\quad \text{a.e. in } (0,T) \times \Omega_k.
	\end{equation*}
	Setting $\pi_k = |\tau_k|$, we arrive at
	\begin{equation}
		\tau_k = \pi_k\,\partial_x u_k,\quad \pi_k\geq 0,\quad \pi_k(1-|\partial_x u_k|)=0.
	\end{equation}
	Therefore, we have derived the limiting system~\eqref{PDE,k}.
	 
	\smallskip
	\noindent
	{\bf Step~7: Improving the convergence of density to $C([0,T];L^{r}(\Omega_k))$.}
    For $r>1$, the renormalised continuity equation reads
	\begin{equation*}
		\partial_t(\rho_{p,k}^r)+\partial_x(\rho_{p,k}^r u_{p,k})+(r-1)\rho_{p,k}^r\partial_x u_{p,k}=0.
	\end{equation*}
	Integrating over $\Omega_k$ and using the Dirichlet boundary condition $\left.u_{p,k}\right|_{\partial\Omega_k}=0$, we obtain
	\begin{equation*}
		\frac{\mathrm{d}}{\mathrm{d}t}\int_{\Omega_k}\rho_{p,k}^r(t,x)\,\mathrm{d}x
		=-(r-1)\int_{\Omega_k}\rho_{p,k}^r(t,x)\partial_x u_{p,k}(t,x)\,\mathrm{d}x.
	\end{equation*}
	From the Proposition \ref{prop: rho dotu} and H\"{o}lder's inequality, we deduce that
	\begin{align*}
		\int_{\Omega_k}|\partial_x u_{p,k}|\,\mathrm{d}x &\leq |\Omega_k|^{1-\frac{1}{p}}\left(\int_{\Omega_k}|\partial_x u_{p,k}|^p\,\mathrm{d}x\right)^{\frac{1}{p}},
	\end{align*}
	which is uniformly bounded in $p$ thanks to the energy estimate. Hence, for any $r\in[1,\infty[$,
	\begin{align*}
		\text{$\left\{\frac{\mathrm{d}}{\mathrm{d}t}\int_{\Omega_k}\rho_{p,k}^r\,\mathrm{d}x\right\}$  is uniformly bounded in $p$.}
	\end{align*}
	Since $\rho_{p,k}\to \rho_k$ strongly in $L^r((0,T)\times \Omega_k)$, the Arzel\`a--Ascoli theorem and the uniform convexity of $L^r$ yield the strong convergence in $C([0,T];L^r(\Omega_k))$:
	\begin{equation}
		\rho_{p,k}\longrightarrow \rho_k \quad \text{strongly in } C([0,T];L^r(\Omega_k)).
	\end{equation}
	This completes the proof of Theorem \ref{thm: p_limit}.
\end{proof}

	\section{Proof of Theorem \ref{thm: main}}\label{sec: proof of thm main}
	We now turn to the proof of our Main Theorem~\ref{thm: main} on the existence of weak solutions to the limiting PDE system, whose spatial domain is the whole of $\mathbb{R}$. Note that the estimates in Section~\ref{sec: unif est for truncated} depend on $k \sim$ the size of the truncated spatial domain, and that the limiting processes $p \to \infty$ and $k \to \infty$ do not commute in general. Thus, extra care is needed when passing to the limits.

	As will be shown in this section, on the unbounded domain $\mathbb{R}$, the Cauchy stress $\tau_k$ will only converge as Radon measures. The identification of $\pi$ requires an energy comparison argument on each compact set, again motivated by \cite{Bresch2026}.
	\begin{lemma}[Local uniform estimates]\label{lem: local es}
		Fix any compact set $K\Subset\mathbb{R}$. Let $(\rho_k,u_k)$ be the solution obtained in Theorem \ref{thm: p_limit}.
		Under the hypotheses of Theorem \ref{thm: main}, there exists a constant $C = C(T,K,\gamma,\mu,a,E_0,c_1(K),c_2)$ 
		such that for all sufficiently large $k$ with $K\subset\Omega_k$, the following estimates hold
		\begin{align}
			& \label{es: local rho}
			\frac{c_1(K) e^{-2t}}{\max\left\{1,\left(\frac{a}{\mu}\right)^{\frac{1}{\gamma}}c_2\right\}}\leq \rho_{k}(x,t)\leq \sup\limits_{x\in K}\rho_{k,0}(x)\exp\left[\frac{3}{2}t+C\right],\quad \text{a.e. on $K\times(0,T)$,}\\
			&\label{es: local energy}
			\sup_{t\in(0,T)}\left\{\int_{K}\left[\frac{1}{2}\rho_{k}\left|u_{k}\right|^2+\frac{a}{\gamma-1}\rho_{k}^{\gamma}\right]\,\mathrm{d}x + \mu \int_0^t\int_{K}|\tau_k|\mathrm{d}x\,\mathrm{d}s\right\} \nonumber \\
			&\leq E_{0}:=\int_{\mathbb{R}}\left[\frac{1}{2}\rho_{0}\left|u_{0}\right|^2+\frac{a}{\gamma-1}\rho_{0}^{\gamma}\right]\,\mathrm{d}x,\\
			&\label{es: local u_t}
			\|\partial_t u_k\|_{L^{1}(0,T;H^{-2}(K))}\leq C.
		\end{align}
	\end{lemma}
	\begin{proof}
		Since $K\Subset\Omega_k$, inequality \eqref{es: u_x} in Proposition \ref{prop: bound for stress and rho} implies that
		\begin{equation*}
			\partial_x u_{p,k}\leq \left(\frac{a}{\mu}\rho_{p,k}^{\gamma}(x,t)+1\right)^{\frac{1}{p-1}}\quad \text{for a.e. } x\in K.
		\end{equation*}
		Following the proof of lower bound for density in Proposition \ref{prop: bound for stress and rho}, we obtain that
		\begin{equation*}
			\rho_{p,k}(t,x)\geq \frac{c_1(K) e^{-2t}}{\max\left\{1,\left(\frac{a}{\mu}\right)^{\frac{1}{\gamma}}c_2\right\}} \quad\text{for a.e. }(t,x)\in(0,T)\times K,
		\end{equation*}
		where $c_1(K), c_2$ are lower and upper bounds of the initial data. For the upper bound, we again repeat the arguments in Proposition \ref{prop: bound for stress and rho}, replacing the domain $\Omega_k$ by the compact set $K$ and the length factor $|\Omega_k|$ by the Lebesgue measure $|K|$. This gives
		\begin{equation*}
			\rho_{p,k}\leq \sup\limits_{K}\rho_{p,k,0}(x)\exp\left[\frac{3}{2}t+C\right].
		\end{equation*}
		We may now conclude \eqref{es: local rho} by sending $p\rightarrow \infty$.
		
		Next, we restrict the energy equality \eqref{es: energy k} to $K$. As $\tau_k\partial_x u_k=|\tau_k|$ \emph{a.e.}, we obtain
		\begin{equation*}
			\int_K\left(\frac{1}{2}\rho_k u_k^2+\frac{a}{\gamma-1}\rho_k^\gamma\right)(t)\,\mathrm{d}x
			+\mu\int_0^t\int_K |\tau_k|\,\mathrm{d}x\,\mathrm{d}s
			\leq E_{k,0}\leq E_0.
		\end{equation*}
		The second inequality $E_{k,0}\leq E_0$ was proved in Proposition \ref{prop: energy}. This proves \eqref{es: energy k}.
		
		Finally, the momentum equation reads
		\begin{equation*}
			\partial_t(\rho_k u_k)=-\partial_x(\rho_k u_k^2)+\mu\partial_x\tau_k-a\partial_x\rho_k^\gamma.
		\end{equation*}
		From the density bounds \eqref{es: local rho} and energy estimates \eqref{es: local energy}, we have $u_k\in L^{\infty}(0,T;L^2(K))$. Together with $|\partial_x u_k|\leq 1$, this implies $u_k\in L^{\infty}(0,T;H^1(K))$. For any test function $\varphi\in C_c^{\infty}(\mathbb{R})$ with $\operatorname{supp}\varphi \subset K$, we have
		\begin{align*}
			\left|\langle\partial_x (\rho_k u_k^2),\varphi\rangle \right|=&\left|\int_{K}(\rho_ku_k^2) \partial_x \varphi\,\mathrm{d}x\right|\\
			\leq &\left\|\rho_k\right\|_{L_x^{\infty}(K)}\left\|u_k\right\|_{L_x^{\infty}(K)}\left\|u_k\right\|_{L_x^{2}(K)} \left\|\partial_x \varphi\right\|_{L_x^{2}}\\
			\leq & C(E_0,K)\left\|\varphi\right\|_{H_x^1},
		\end{align*}
		so $\partial_x (\rho_k u_k^2) \in L^{\infty}(0,T;H^{-1}(K))$ uniformly in $k$, and consequently in $L^{2}(0,T;H^{-1}(K))$ uniformly in $k$.
		Similarly, since $\rho_k$ is bounded by\eqref{es: local rho}
		\begin{align*}
			\left|\langle\partial_x \rho_k^{\gamma},\varphi\rangle \right|=&\left|\int_{K}\rho_k^{\gamma} \partial_x \varphi\,\mathrm{d}x\right|\\
			\leq &\left\|\rho_k\right\|_{L_x^{\infty}(K)}^{\gamma}|K|^{\frac{1}{2}} \left\|\partial_x \varphi\right\|_{L_x^{2}}\\
			\leq & C(E_0,K)\left\|\varphi\right\|_{H_x^1},
		\end{align*}
		hence $\partial_x \rho_k^{\gamma}$ is uniformly bounded in $L^{\infty}(0,T;H^{-1}(K))\hookrightarrow L^{2}(0,T;H^{-1}(K))$, independent of $k$. 
		
		For $\partial_x \tau_k$, recall that $\tau_{p,k} \rightharpoonup \tau_k$ weakly in $L^1((0,T)\times(\Omega_k))$. Then
		\begin{align}
			\label{es: tau L1}
			\int_{0}^{T}\int_{K}|\tau_k|\,\mathrm{d}x \,\mathrm{d}t &\leq \liminf_{p\rightarrow \infty}\int_{0}^{T}\int_{K}|\tau_{p,k}|\,\mathrm{d}x \,\mathrm{d}t\nonumber\\
			&\leq \liminf_{p\rightarrow\infty}\int_{0}^{T}\int_{K}|\partial_{x}u_{p,k}|^{p-1}\,\mathrm{d}x \,\mathrm{d}t\\
			&\leq \liminf_{p\rightarrow \infty}\left(\frac{p-1}{p}\int_{0}^{T}\int_{K}|\partial_{x}u_{p,k}|^{p}\,\mathrm{d}x \,\mathrm{d}t+\frac{1}{p}T\times |K|\right)\nonumber\\
			&\leq C(E_0,K,T),\nonumber
		\end{align}
		which implies $\tau_k\in L^{1}((0,T)\times K)$ uniformly in $k$. By Sobolev embedding, $\tau_k\in L^{1}(0,T;H^{-1}(K))$ and hence $\mu\partial_x\tau_k\in L^{1}(0,T;H^{-2}(K))$ uniformly in $k$. In addition, $L^{2}(0,T;H^{-1}(K))\hookrightarrow L^{1}(0,T;H^{-2}(K))$ continuously, so
		\begin{equation*}
			\|\partial_t(\rho_k u_k)\|_{L^{1}(0,T;H^{-2}(K))}\leq C(K).
		\end{equation*}
		Moreover, for any test function $\varphi\in C_c^{\infty}(\mathbb{R})$ with $\operatorname{supp}\varphi \subset K$,
		\begin{align*}
			\left|\langle\partial_x (\rho_k u_k),\varphi\rangle \right|=&\left|\int_{K}(\rho_ku_k) \partial_x \varphi\,\mathrm{d}x\right|\\
			\leq &\left\|\rho_k\right\|_{L_x^{\infty}(K)}^{\frac{1}{2}}\|\rho_k^{\frac{1}{2}}u_k\|_{L_x^{2}(K)}\left\|\partial_x \varphi\right\|_{L_x^{2}}\\
			\leq & C(E_0,K)\left\|\varphi\right\|_{H_x^1},
		\end{align*}
		which shows that $\partial_x (\rho_k u_k)\in L^{\infty}(0,T;H^{-1}(K))$ uniformly in $k$, and hence in $L^{2}(0,T;H^{-1}(K))$. Together with the continuity equation $\partial_t\rho_k=-\partial_x(\rho_k u_k)$, we deduce
		\begin{equation*}
			\|\partial_t\rho_k\|_{L^2(0,T;H^{-1}(K))}\leq C(E_0, K).
		\end{equation*}
		Finally, since $\rho_{k}$ has a positive bound (depending only on $K$), we may express 
		\begin{equation*}
			\partial_t u_k = \frac{1}{\rho_k}\partial_t(\rho_k u_k) - \frac{u_k}{\rho_k}\partial_t\rho_k,
		\end{equation*}
		where the coefficients $\frac{1}{\rho_k}$ and $\frac{u_k}{\rho_k}$ belong to $L^\infty((0,T)\times K)$ uniformly in $k$. Thus
		\begin{equation*}
			\|\partial_t u_k\|_{L^{1}(0,T;H^{-2}(K))}\leq C(E_0,K). 
		\end{equation*}
		This completes the proof of \eqref{es: local u_t}
	\end{proof}

	Finally, we are at the stage of proving out main result of the paper.
	
	\begin{proof}[Proof of Theorem \ref{thm: main}]
		
		We divide our proof into four steps.
		
		\smallskip
		\noindent
		{\bf Step~1: Compactness.} Let $K\Subset \mathbb{R}$ be an arbitrary compact set. By Lemma~\ref{lem: local es}, for all sufficiently large $k$ so that $K\subset \Omega_k$, we have
		\begin{align*}
			&\|u_k\|_{L^2(0,T;H^1(K))} + \|\partial_t u_k\|_{L^{1}(0,T;H^{-2}(K))} \leq C(K,T),\\
			&\|\rho_k\|_{L^\infty((0,T)\times K)} + \|\partial_t\rho_k\|_{L^2(0,T;H^{-1}(K))} \leq C(K,T),
		\end{align*}
		where the constants depend only on $T$, $K$, $E_0$ and the initial data, but not on $k$.
		
		Set $X=H^1(K)$, $Y=L^2(K)$, $Z=H^{-2}(K)$. The embedding $X \hookrightarrow \hookrightarrow Y$ is compact and $Y\hookrightarrow Z$ is continuous. Since $\{u_k\}$ is bounded in $L^2(0,T;X)$ and $\{\partial_t u_k\}$ is bounded in $L^1(0,T;Z)$, we may apply the Aubin--Lions--Simon lemma (Theorem \ref{thm:ALS}) to deduce that $\{u_k\}$ is relatively compact in $L^2(0,T;Y)$. Hence, up to a subsequence (not relabelled), we deduce that 
		\begin{equation*}
			u_k\rightarrow u \quad \text{strongly in}\quad  L^2((0,T)\times K).
		\end{equation*}
        Since $\rho_k$ is uniformly bounded in $L^\infty((0,T)\times K)$, it is uniformly bounded 
        in $L^2(0,T; L^q(K))$ for every $1\leq q<\infty$. Since the embedding $L^q(K)\hookrightarrow\hookrightarrow L^2(K)\, (q>2)$ is compact, applying the Theorem \ref{thm:ALS} once again with $X=L^q(K)$, $Y=L^2(K)$ and $Z=H^{-1}(K)$, we obtain a subsequence (not relabelled) such that 
		\begin{equation*}
			\rho_k\rightarrow\rho \quad\text{strongly in } L^2((0,T)\times K).
		\end{equation*}
        By interpolation with the uniform $L^\infty$ bound, we upgrade this to
        \begin{equation*}
            \rho_k \longrightarrow \rho \quad \text{strongly in } L^r((0,T)\times K) \quad \text{for every } 1\leq r<\infty.
        \end{equation*}
		%From the continuity equation $\partial_t \rho_k=-\partial_x(\rho_ku_k)$ and uniform bounds established above, we have $\partial_t \rho_k\in L^{\infty}(0,T;H^{-1}(K))$. Since the embedding $L^r(K)\hookrightarrow\hookrightarrow H^{-1}(K)$ is compact, applying the Theorem \ref{thm:ALS} once again with $X=L^r(K)$, $Y=H^{-1}(K)$ and $p=\infty$, we obtain a subsequence (not relabelled) such that 
		%\begin{equation*}
		% \rho_k\rightarrow\rho \quad\text{strongly in } C([0,T];H^{-1}(K)).
		%\end{equation*}
		Let $K_m = [-m,m]\Subset \mathbb{R}$. For each fixed $m$, by the compactness obtained above, there exists a subsequence $\left\{k^{(m)}_j\right\}_{j=1}^\infty$ such that
		\begin{equation*}
			u_{k^{(m)}_j} \rightarrow u^{(m)} \quad\text{in } L^2((0,T)\times K_m),\qquad
			\rho_{k^{(m)}_j} \rightarrow \rho^{(m)} \quad\text{in } L^r((0,T)\times K_m),\quad r\in[1,\infty[.
		\end{equation*}
		By diagonalisation (namely, choosing $k_j = k_j^{(j)}$) and uniqueness of limits, we deduce the existence of functions $\rho$, $u$, defined globally on $\mathbb{R}$, such that  $u^{(m)} = u|_{K_m}$ and $\rho^{(m)} = \rho |_{K_m}$ for all $m$. Along this diagonal subsequence, the following convergence results hold for every $m$:
		\begin{align*}
			u_{k}\rightarrow u \quad\text{strongly in } L^2((0,T)\times K_m),\qquad
			\rho_{k}\rightarrow \rho \quad\text{strongly in } L^r((0,T)\times K_m),\quad r\in[1,\infty[.
		\end{align*}
		Moreover, the uniform $L^1$ bound \eqref{es: tau L1} for $\{\tau_k\}$ implies the existence of a subsequence such that $\tau_{k}$ converges to $\tau$ in the distributional sense, i.e.,
		\begin{equation*}
			\int_0^T\int_{\mathbb{R}} \tau_{k}\,\varphi\,\mathrm{d}x\,\mathrm{d}t \rightarrow \int_0^T\int_{\mathbb{R}} \tau\,\varphi\,\mathrm{d}x\,\mathrm{d}t \qquad\text{for every } \varphi\in C_c^\infty((0,T)\times\mathbb{R}).
		\end{equation*}
		Restricting to continuous test functions, we deduce that $\tau_{k}\stackrel{*}{\rightharpoonup}\tau$ as Radon measures.
		
		\smallskip
		\noindent
		{\bf Step~2: Passage to the limits.}
		By Theorem \ref{thm: p_limit}, $|\partial_x u_k|\leq 1$ \emph{a.e.} in $(0,T)\times\Omega_k$.
		In particular, for any compact set $K_m\Subset\Omega_k$, we have $\|\partial_x u_k\|_{L^\infty((0,T)\times K_m)} \leq 1$. By the Banach--Alaoglu theorem, there exists a subsequence (still denoted by $k$) such that $\partial_x u_k \stackrel{*}{\rightharpoonup} g $ in $L^\infty((0,T)\times K_m)$. Since $u_k \rightarrow u \text{ strongly in } L^2((0,T)\times K_m)$, for any test function $\phi\in C_c^{\infty}((0,T)\times K_m)$, we have
		\begin{align*}
			\int_0^T\int_{K_m}u\,\partial_x\phi\,\mathrm{d}x\,\mathrm{d}t=&\lim_{k\rightarrow\infty}\int_0^T\int_{K_m}u_k\,\partial_x\phi\,\mathrm{d}x\,\mathrm{d}t=-\lim_{k\rightarrow\infty}\int_0^T\int_{K_m}\partial_x u_k \phi\,\mathrm{d}x\,\mathrm{d}t\\=&-\int_0^T\int_{K_m}g\, \phi\,\mathrm{d}x\,\mathrm{d}t,
		\end{align*}
		so that $g=\partial_x u$ \emph{a.e.} and hence $\partial_x u_k \stackrel{*}{\rightharpoonup} \partial_x u$ in $L^\infty((0,T)\times K_m)$.
		
		Therefore, from the uniform estimates from Lemma \ref{lem: local es} and the Theorem \ref{thm:ALS}, we deduce that
		\begin{align}\label{converge k}
			u_k &\longrightarrow u \quad \text{strongly in } L^2((0,T)\times K_m),\nonumber\\
			\partial_x u_k &\rightharpoonup \partial_x u \quad \text{weakly-* in } L^\infty((0,T)\times K_m),\nonumber\\
			\tau_k &\stackrel{*}{\rightharpoonup} \tau \quad \text{as Radon measures},\\
			\rho_k &\rightharpoonup\rho \quad \text{weakly-* in } L^\infty((0,T)\times K_m),\nonumber\\
            \rho_{k}&\rightarrow \rho \quad\text{strongly in } L^r((0,T)\times K_m),\quad r\in[1,\infty[.\nonumber
			%\rho_k &\longrightarrow \rho \quad \text{strongly in } C([0,T];H^{-1}(K_m)).\nonumber
		\end{align}
        Therefore, taking any test function $\varphi\in C_c^\infty((0,T)\times\mathbb{R})$ and choosing $m$ such that $\operatorname{supp}\varphi\subset(0,T)\times(-m,m)$,
		we can pass to the limits in the weak formulation of the continuity and momentum equations, obtaining that
		\begin{align*}
			&\int_0^T\int_{\mathbb{R}} (\rho\partial_t\varphi+\rho u\partial_x\varphi)\,\mathrm{d}x\,\mathrm{d}t=0,\\
			&\int_0^T\int_{\mathbb{R}} (\rho u\partial_t\varphi+\rho u^2\partial_x\varphi+\mu\tau\partial_x\varphi+a\rho^\gamma\partial_x\varphi)\,\mathrm{d}x\,\mathrm{d}t=0,
		\end{align*}
		In other words, $(\rho,u,\tau)$ satisfies the limit system in the distributional sense.
		
		Moreover, from $\partial_x u_k\stackrel{*}{\rightharpoonup} \partial_x u$ in  $L^\infty((0,T)\times K_m)$ and the bound $|\partial_x u_k|\leq 1$, we deduce by weak-* lower semi-continuity in $L^{\infty}$ that
		\begin{equation*}
			|\partial_x u|\leq 1 \quad \text{a.e. in }(0,T)\times K_m, 
		\end{equation*}
		From the energy estimates \eqref{es: local energy} and uniform bounds for $\rho_k$, we obtain that $u_k\in L^{\infty}(0,T;L^2(K_m))$ uniform in $k$. Since $u_k\rightarrow u$ strongly in $L^2(0,T;L^2(K_m))$, there is a subsequence (still denoted by $u_k$) that converges weakly-* in $L^{\infty}(0,T;L^2(K_m))$ to some limit. By uniqueness of strong convergence, this limit coincides with $u$.  Furthermore, the same energy estimate also provides a uniform $L^{\infty}$ bound for $\partial_x u$, so $u\in L^\infty(0,T;H^1(K_m)) \emb  L^{\infty}((0,T)\times K_m)$.

		\smallskip
		\noindent
		{\bf Step~3: Identification of the limiting stress via energy comparison.}
		For a fixed compact set $K$, choose $\phi \in C_c^{\infty}(\mathbb{R})$ with $\phi \equiv 1$ on $K$, and $\operatorname{supp} \phi \subset \Omega_k$ for sufficiently large $k$. Recall from Theorem~\ref{thm: p_limit} that $u_{p,k} \rightharpoonup u_k$ weakly in $H^1(\Omega_k)$, hence the boundary condition $u_k|_{\partial \Omega_k} = 0$ holds for $u_k$. As a consequence, $\phi u_k \in H_0^1(\Omega_k)$ can be used as a test function in the momentum equation \eqref{PDE,k}
		\begin{align*}
			&\frac{\mathrm{d}}{\mathrm{d}t}\int_{\Omega_k} \phi\left(\frac{1}{2}\rho_k u_k^2+\frac{a}{\gamma-1}\rho_k^\gamma\right)\,\mathrm{d}x
			+\mu\int_{\Omega_k}\phi |\tau_k|\,\mathrm{d}x \\
			&\qquad = \int_{\Omega_k} \rho_k u_k^2\,\partial_x\phi\,\mathrm{d}x
			+\mu\int_{\Omega_k}\tau_k u_k \partial_x\phi \,\mathrm{d}x
			-a\int_{\Omega_k}\rho_k^\gamma u_k \partial_x\phi\,\mathrm{d}x,
		\end{align*}
		where we have employed the identity $\tau_k\partial_x u_k=|\tau_k|$ \emph{a.e.} in $\Omega_k$. Then, integration over time leads to
		\begin{align}\label{ineq: local energy k}
			&\int_K\left(\frac{1}{2}\rho_k u_k^2+\frac{a}{\gamma-1}\rho_k^\gamma\right)(t,x)\,\mathrm{d}x
			+\mu\int_0^t\int_K |\tau_k|\,\mathrm{d}x\,\mathrm{d}s \\
			\leq& \int_{\Omega_k}\phi\left(\frac{1}{2}\rho_{k,0} u_{k,0}^2+\frac{a}{\gamma-1}\rho_{k,0}^\gamma\right)\,\mathrm{d}x
			+ C(K)\int_0^t\int_{\operatorname{supp}\partial_x\phi} (|u_k|^2+|u_k|)\,\mathrm{d}x\,\mathrm{d}s\nonumber\\
			&+\mu\int_0^t\int_{\operatorname{supp}\partial_x\phi} \tau_k u_k \partial_x\phi \,\mathrm{d}x\,\mathrm{d}s,\nonumber
		\end{align}
		where the constant $C(K)$ relies on $\|\partial_x\phi\|_{L_x^\infty}$ and the uniform boundedness of $\rho_k$.
		
		The strong convergence $u_k\rightarrow u$ in $L^2((0,T)\times K)$ implies that
		\begin{equation*}
			\int_0^t\int_{\operatorname{supp}\partial_x\phi} \left(|u_k|^2+|u_k|\right)\,\mathrm{d}x\,\mathrm{d}s
			\longrightarrow
			\int_0^t\int_{\operatorname{supp}\partial_x\phi} \left(|u|^2+|u|\right)\,\mathrm{d}x\,\mathrm{d}s.
		\end{equation*}
		Also, as $u_k\partial_x\phi \to u\partial_x\phi$ strongly in $L^2$ with compact support and $\tau_k \stackrel{*}{\rightharpoonup} \tau$ as Radon measures, we have that
		\begin{equation*}
			\int_0^t\int_{\operatorname{supp}\partial_x\phi} \tau_k u_k \partial_x\phi\,\mathrm{d}x\,\mathrm{d}s
			\longrightarrow
			\int_0^t\int_{\operatorname{supp}\partial_x\phi} \tau\, u\, \partial_x\phi\,\mathrm{d}x\,\mathrm{d}s.
		\end{equation*}
		Thus, taking limsup in \eqref{ineq: local energy k} leads to
		\begin{equation}\label{eq:local_energy_phi}
			\begin{aligned}
				&\int_K\left(\frac{1}{2}\rho u^2+\frac{a}{\gamma-1}\rho^\gamma\right)(t,x)\,\mathrm{d}x
				+\mu\int_0^t\int_K |\tau|\,\mathrm{d}x\,\mathrm{d}s \\
				\leq& \int_{\Omega_k}\phi\left(\frac{1}{2}\rho_0 u_0^2+\frac{a}{\gamma-1}\rho_0^\gamma\right)\,\mathrm{d}x
				+ C(K)\int_0^t\int_{\operatorname{supp}\partial_x\phi} \left(|u|^2+|u|\right)\,\mathrm{d}x\,\mathrm{d}s\\
				&+\mu\int_0^t\int_{\operatorname{supp}\partial_x\phi} \tau\, u\, \partial_x\phi\,\mathrm{d}x\,\mathrm{d}s.
			\end{aligned}
		\end{equation}
		We now choose a sequence of cut-off functions $\{\phi_m\}$ such that $\phi_m \equiv 1$ on $K$, $0 \leq \phi_m \leq 1$, $\phi_m \to \mathbb{I}_K$ pointwise, and $\operatorname{supp} \partial_x\phi_m \to \partial K$ with $\| \partial_x\phi_m\|_{L^1} \leq C$ uniformly. Since $\tau \in L_{\mathrm{loc}}^1((0,T)\times \mathbb{R})$ by \eqref{es: tau L1} and $u \in L^{\infty}((0,T)\times K)$, the product $\tau u$ is an $L_{\mathrm{loc}}^1$-function. Moreover, the size of $\operatorname{supp}\partial_x\phi_m$ tends to zero as $m \to \infty$. Then, in view of the absolute continuity of the integral, we obtain that
		\begin{equation*}
			\int_0^t\int_{\operatorname{supp}\partial_x\phi_m} \tau u\,\partial_x\phi_m\,\mathrm{d}x\,\mathrm{d}s \longrightarrow 0,
		\end{equation*}
		and similarly
		\begin{equation*}
			\int_0^t\int_{\operatorname{supp}\partial_x\phi_m} (|u|^2+|u|)\,\mathrm{d}x\,\mathrm{d}s \longrightarrow 0.
		\end{equation*}
		All the remaining terms on the right-hand side of \eqref{eq:local_energy_phi} vanish as $m\to\infty$. Hence, we arrive at the local energy inequality:
		\begin{equation}\label{es: energy klimit}
			\int_K\left(\frac{1}{2}\rho u^2+\frac{a}{\gamma-1}\rho^\gamma\right)(t,x)\,\mathrm{d}x
			+\mu\int_0^t\int_K |\tau|\,\mathrm{d}x\,\mathrm{d}s
			\leq \int_{K}\left(\frac{1}{2}\rho_0 u_0^2+\frac{a}{\gamma-1}\rho_0^\gamma\right)\,\mathrm{d}x.
		\end{equation}
		On the other hand, since the limit momentum equation holds in the distributional sense, we may test it with $\phi u$, where $\phi\in C_c^\infty(\mathbb R)$. Repeat the same argument as for \eqref{es: energy klimit}, we obtain the energy equality for $(\rho,u)$:
		\begin{equation}\label{es: energy}
			\int_K\left(\frac{1}{2}\rho u^2+\frac{a}{\gamma-1}\rho^\gamma\right)(t)\,dx
			+\mu\int_0^t\int_K \tau\,\partial_x u\,\mathrm{d}x\,\mathrm{d}s
			= \int_{K}\left(\frac{1}{2}\rho_0 u_0^2+\frac{a}{\gamma-1}\rho_0^\gamma\right)\,\mathrm{d}x.
		\end{equation}
		Now, one deduces from comparing \eqref{es: energy klimit} and \eqref{es: energy} that
		\begin{equation*}
			\int_0^t\int_K |\tau|\,\mathrm{d}x\,\mathrm{d}s
			\leq \int_0^t\int_K \tau\,\partial_x u\,\mathrm{d}x\,\mathrm{d}s .
		\end{equation*}
		As it is proved that $|\partial_x u|\leq 1$ \emph{a.e.}, we obtain that $\tau\partial_x u\leq |\tau|$ \emph{a.e.}, so 
		\begin{equation*}
			\int_0^t\int_K (|\tau|-\tau\partial_x u)\,\mathrm{d}x\,\mathrm{d}s = 0.
		\end{equation*}
		Thus,
		\begin{equation*}
			|\tau| = \tau\partial_x u \quad \text{a.e. in } (0,T)\times K.
		\end{equation*}
		Setting $\pi:=|\tau|$, we obtain that $\tau=\pi\partial_x u$, $\pi\geq 0$, and $\pi(1-|\partial_x u|)=0$ \emph{a.e.} in $K$. As $K$ is arbitrary, we may conclude \eqref{cdt} on $\mathbb{R}$.
		
		\smallskip
		\noindent
		{\bf Step~4: Strong convergence of the density in $C([0,T];L_{\mathrm{loc}}^r(\mathbb{R}))$.}
        From Step~3 above, we know
        \begin{equation*}
			\rho_k\longrightarrow \rho \quad \text{strongly in}\quad  L^r((0,T)\times K),\quad \text{for every }1\leq r<\infty.
		\end{equation*}
        Thus, up to an unrelabelled subsequence, $\rho_k \to \rho$ \emph{a.e.} on $K\times(0,T)$, and hence
		\begin{equation*}
			0<C_1(K)\leq \rho(t,x)\leq C_2(K)\quad\text{a.e. }(t,x)\in(0,T)\times K,
		\end{equation*}
		with $C_1(K), C_2(K)$ depending only on $T, \gamma, a, \mu, E_0$ and the compact set $K$. 
		
		Finally, we show that the convergence of $\rho_k\rightarrow \rho$ holds in $C((0,T);L^r(K))$ for any $1\leq r<\infty$. For the approximate solution $\rho_k$, the renormalised continuity equation on $\mathbb{R}$ reads
		\begin{equation*}
			\partial_t(\rho_k^r)+\partial_x(\rho_k^r u_k)+(r-1)\rho_k^r\partial_x u_k=0.
		\end{equation*}
		Integrating over $K$ and applying the divergence theorem, we obtain
		\begin{equation*}
			\frac{\mathrm{d}}{\mathrm{d}t}\int_{K}\rho_k^r(t,x)\,\mathrm{d}x
			=-(r-1)\int_{K}\rho_k^r\partial_x u_k(t,x)\,\mathrm{d}x-\left.\rho_k^r u_k\right|_{\partial K}.
		\end{equation*}
		From the energy bound \eqref{es: local energy} and $|\partial_x u_k|\leq 1$, we have $\|u_k\|_{L^\infty(0,T;H^1(K))}\leq C(K)$. By the one-dimensional Sobolev embedding $H^1(K)\hookrightarrow L^{\infty}(K)$, we know $\|u_k\|_{L^\infty((0,T)\times K)}\leq C(K)$. Hence, the boundary term is uniformly bounded
		\begin{equation*}
			\left\||\left.\rho_k^r u_k\right|_{\partial K}\right\|_{L^{\infty}}\leq C(K). 
		\end{equation*}
		Therefore, using the upper bound for $\rho_k$ and $|\partial_x u_k|\leq 1$ \emph{a.e.},
		\begin{equation}
			\left|\frac{\mathrm{d}}{\mathrm{d}t}\|\rho_k(t)\|_{L_x^r(K)}^r\right|
			\leq (r-1) \|\rho_k\|_{L_x^{\infty}(K)}^r|K|+C(K)
			\leq C(T,\gamma,\mu,a, E_0,K),
		\end{equation}
		which is uniform in k. Thus the functions $t\mapsto \|\rho_k(t)\|_{L^r(K)}^r$ are equi-(Lipschitz) continuous on $[0,T]$.
		
		The argument in that subsection relies only on the energy estimates, the \emph{a.e.} bound $|\partial_x u_k|\leq 1$, 
		and the uniform boundedness of $\rho_k$, all of which have been established uniformly in $k$ on $K$. 
		Therefore, we can repeat the same density compactness argument as in Step~6 of the proof of Theorem \ref{thm: p_limit}, with $\Omega_k$ replaced by $K$, $(\rho_{p,k},u_{p,k})$ replaced by $(\rho_k,u_k)$, and $(\rho_k,u_k)$ replaced by $(\rho,u)$. Thus, $\rho_k \to \rho$ in $C([0,T]; L^r(K))$ for each $r \in [1,\infty[$. By the arbitrariness of $K$, we conclude that
		\begin{equation}
			\rho_k\longrightarrow \rho \quad \text{strongly in } C([0,T];L^r_{\mathrm{loc}}(\mathbb{R}))\quad \text{for every }1\leq r<\infty.
		\end{equation}
		This completes the proof of Theorem \ref{thm: main}.
	\end{proof}

    \bigskip
	\appendix
	\section{Detailed A Priori Estimates}\label{sec: appendix}
	In this appendix, we provide detailed proofs of Propositions \ref{prop: bound for stress and rho}, \ref{prop: rho dotu}, and \ref{prop: bound tau_pk}, adapted from Bresch--Burtea--Szlenk \cite{Bresch2026} for the periodic setting. We present the modifications required by the Dirichlet boundary conditions and the dependence on the truncation parameter $k$.
	\begin{proof}[Proof of Proposition \ref{prop: bound for stress and rho}]
		Following \cite{Bresch2026}, we set $\sigma_{p,k}=\mu|\partial_x u_{p,k}|^{p-2}\partial_x u_{p,k}-a\rho_{p,k}^\gamma$. We shall prove that for all $(t,x)\in[0,T]\times\overline{\Omega}_k$,
		\begin{equation*}
			\sigma_{p,k}(t,x)\leq \sup_{x\in\Omega_k}\sigma_{p,k}(0,x)\leq \mu.
		\end{equation*}
		The argument combines the classical parabolic maximum principle in the interior with a separate analysis on the boundary $\partial\,\Omega_k$.

        As noted in the proof of Proposition \ref{prop: bound for stress and rho}, the two trivial cases (when $\sigma_{p,k}$ is constant in space or its maximum is non-positive) are immediate. We therefore assume that $\sigma_{p,k}$ is non-constant with a strictly positive maximum.
		%We first consider two trivial cases. If $\sigma_{p,k}$ is constant in space, then the desired bound follows immediately from the initial condition, since $\sigma_{p,k}(t,x)=\sigma_{p,k}(0,x)\leq \mu$ for all $t\geq 0$. On the other hand, if the maximum value of $\sigma_{p,k}$ is non-positive, then the conclusion $\sigma_{p,k}\leq \mu$ is trivially satisfied because $\mu>0$. Hence, in the sequel, we may assume that $\sigma_{p,k}$ is not constant and its maximum is strictly positive.
		
		Differentiating the momentum equation $\eqref{PDE,pk}_2$ and using the continuity equation $\eqref{PDE,pk}_1$, we obtain after simplification
		\begin{equation}\label{eq: sigma_pk1}
			\partial_t\partial_xu_{p,k}+u_{p,k}\partial_{xx}^2 u_{p,k}+\left(\partial_x u_{p,k}\right)^2-\partial_x\left(\frac{1}{\rho_{p,k}}\partial_x\left(\mu\left|\partial_x u_{p,k}\right|^{p-2}\partial_x u_{p,k}-\rho_{p,k}^{\gamma}\right)\right)=0.
		\end{equation}
		The function $H(s):=\left|s\right|^{p-2}s$ is strictly increasing with
		\begin{equation*}
			H^{\prime}(s)=(p-1)\left|s\right|^{p-2}\geq 0
		\end{equation*}
		and its inverse is given by
		\begin{equation}\label{H inverse}
			H^{-1}(s)=\left|s\right|^{p^{\prime}-2}s\quad \text{with} \quad \left(H^{-1}\right)^{\prime}(s)=\left(p^{\prime}-1\right)\left|s\right|^{p^{\prime}-2}\geq 0,
		\end{equation}
		where $p^{\prime}=\frac{p}{p-1}$ is the conjugate exponent of $p$. Note that $p^{\prime}-2=\frac{p}{p-1}-2=-\frac{p-2}{p-1}$. Multiplying \eqref{eq: sigma_pk1} by $\mu H^{\prime}(\partial_x u_{p,k})$ gives
		\begin{equation}\label{eq: H sigma_pk}
			\partial_t\left(\mu H\left(\partial_x u_{p,k}\right)\right)+u_{p,k} \partial_x\left(\mu H\left(\partial_x u_{p,k}\right)\right)+\mu H^{\prime}\left(\partial_x u_{p,k}\right)\left(\partial_x u_{p,k}\right)^2-\mu H^{\prime}\left(\partial_x u_{p,k}\right) \partial_x\left(\frac{1}{\rho_{p,k}} \partial_x \sigma_{p,k}\right)=0 .
		\end{equation}
		We obtain from \eqref{eq: H sigma_pk} that
		\begin{align*}
			& \partial_t \sigma_{p,k}+u_{p,k} \partial_x \sigma_{p,k}+\mu H^{\prime}\left(\partial_x u_{p,k}\right)\left(\partial_x u_{p,k}\right)^2-\mu H^{\prime}\left(\partial_x u_{p,k}\right) \partial_x\left(\frac{1}{\rho_{p,k}} \partial_x \sigma_{p,k}\right) \\
			& =-\left(\partial_t \rho_{p,k}^\gamma+u_{p,k} \partial_x \rho_{p,k}^\gamma\right)=\gamma \rho_{p,k}^\gamma \partial_x u_{p,k}.
		\end{align*}
		Simplifying further, we arrive at
		\begin{align}\label{eq: dt sigma1}
			\partial_t \sigma_{p,k}+u_{p,k} \partial_x \sigma_{p,k}&-\mu H^{\prime}\left(\partial_x u_{p,k}\right) \partial_x\left(\frac{1}{\rho_{p,k}} \partial_x \sigma_{p,k}\right) =\gamma \rho_{p,k}^\gamma \partial_x u_{p,k}-\mu(p-1)\left|\partial_x u_{p,k}\right|^p \nonumber\\
			& =\gamma\left(-\sigma_{p,k}+\mu\left|\partial_x u_{p,k}\right|^{p-2} \partial_x u_{p,k}\right) \partial_x u_{p,k}-\mu(p-1)\left|\partial_x u_{p,k}\right|^p \\
			& =-\gamma \sigma_{p,k} \partial_x u_{p,k}+\mu(1+\gamma-p)\left|\partial_x u_{p,k}\right|^p.\nonumber
		\end{align}
		%	Consequently,
		%	\begin{equation}
			%		\partial_t \sigma_{p,k}+u_{p,k}\partial_x \sigma_{p,k}-\mu H^{\prime}\left(\partial_x u_{p,k}\right)\partial_x\left(\frac{1}{\rho_{p,k}}\partial_x \sigma_{p,k}\right)=-\gamma \sigma_{p,k}\partial_x u_{p,k}+\mu\left(1+\gamma-p\right)\left|\partial_x u_{p,k}\right|^p.
			%	\end{equation}
		Moreover, observe that
		\begin{align}\label{eq: decomposition}
			-\gamma \sigma_{p,k} \partial_x u_{p,k} & =-\frac{\gamma}{\mu^{\frac{1}{p-1}}} \sigma_{p,k} H^{-1}\left(\mu\left|\partial_x u_{p,k}\right|^{p-2} \partial_x u_{p,k}\right) \nonumber\\
			& =-\frac{\gamma}{\mu^{\frac{1}{p-1}}}\left[H^{-1}\left(\sigma_{p,k}+a \rho_{p,k}^\gamma\right)-H^{-1}\left(\sigma_{p,k}\right)\right] \sigma_{p,k}-\frac{\gamma}{\mu^{\frac{1}{p-1}}} H^{-1}\left(\sigma_{p,k}\right) \sigma_{p,k} \nonumber\\
			& =-\frac{\gamma}{\mu^{\frac{1}{p-1}}}\left[H^{-1}\left(\sigma_{p,k}+a \rho_{p,k}^\gamma\right)-H^{-1}\left(\sigma_{p,k}\right)\right] \sigma_{p,k}-\frac{\gamma}{\mu^{\frac{1}{p-1}}}\left|\sigma_{p,k}\right|^{p^{\prime}}.
		\end{align}
		For the first term, since $H^{-1}$ is smooth and increasing, there exists $\xi\in \left(\sigma_{p,k},\sigma_{p,k}+a\rho_{p,k}^\gamma\right)$ such that
		\begin{equation*}
			\left[H^{-1}\left(\sigma_{p,k}+a \rho_{p,k}^\gamma\right)-H^{-1}\left(\sigma_{p,k}\right)\right]= (H^{-1})^{\prime}(\xi)\cdot a \rho_{p,k}^\gamma\geq 0.
		\end{equation*}
		Let $M(t)=\sup_{x\in\Omega_k}\sigma_{p,k}(t,x)$. Since $\sigma_{p,k}$ is continuous and $\Omega_k$ is compact, for each $t$ there exists $x_1(t)\in\overline{\Omega}_k$ such that $M(t)=\sigma_{p,k}(t,x_1(t))$. If $x_1(t)$ lies in the interior for $t>0$, then at $(t,x_1(t))$ we have $\partial_x\sigma_{p,k}=0$ and $\partial_x^2\sigma_{p,k}\leq 0$. Substituting into the evolution equation \eqref{eq: dt sigma1} and using the decomposition \eqref{eq: decomposition}, we obtain
		\begin{align*}
			\frac{\mathrm{d}}{\mathrm{d}t}\sigma_{p,k}(t,x_1)=&\mu H^{\prime}\left(\partial_x u_{p,k}\right)\frac{1}{\rho_{p,k}} \partial_x^2\sigma_{p,k}(t,x_1) -\frac{a\gamma}{\mu^{\frac{1}{p-1}}}(H^{-1})^{\prime}(\xi) \rho_{p,k}^\gamma(t,x_1)\sigma_{p,k}(t,x_1)\\
			&-\frac{\gamma}{\mu^{\frac{1}{p-1}}}\left|\sigma_{p,k}\right|^{p^{\prime}}(t,x_1)+\mu\left(1+\gamma-p\right)\left|\partial_x u_{p,k}\right|^p(t,x_1)\\
			\leq& -\frac{a\gamma}{\mu^{\frac{1}{p-1}}} (H^{-1})^{\prime}(\xi) \rho_{p,k}^\gamma(t,x_1)\sigma_{p,k}(t,x_1)+\mu\left(1+\gamma-p\right)\left|\partial_x u_{p,k}\right|^p(t,x_1).
		\end{align*} 
		For $p$ sufficiently large, i.e., $p\geq 1+\gamma$, the second term of the right-hand side is non-positive. 
		Hence
		\begin{equation*}
			\frac{\mathrm{d}}{\mathrm{d}t}\sigma_{p,k}(t,x_1) \leq -f(t)\sigma_{p,k}(t,x_1),\qquad f(t)=\frac{a\gamma}{\mu^{\frac{1}{p-1}}} (H^{-1})^{\prime}(\xi) \rho_{p,k}^\gamma(t,x_1)\geq 0.
		\end{equation*}
		By Gr\"{o}nwall's inequality, 
		\begin{equation*}
			\sigma_{p,k}(t,x_1)\leq \sigma_{p,k}(0,x_1)\exp\left\{-\int_0^t f(s)\,\dd s\right\}\leq \sigma_{p,k}(0,x_1)\leq \mu.
		\end{equation*}
		Thus $\sigma_{p,k}(t,x)\leq\mu$ for all interior points.
		
		\noindent\textbf{Lower bound for the density.} With the pointwise bound on $\partial_x u_{p,k}$ at hand, we  derive a lower bound for the density using the Lagrangian approach in \cite{Bresch2026}. Define the flow map $X_t(x)$ by
		\begin{equation*}
			\dot{X}_{t}(x) = u_{p,k}(t,X_{t}(x)).
		\end{equation*}
		Along the flow one may estimate
		\begin{align*}
			\partial_{t}\rho_{p,k}(t,X_{t}(x)) &= -\rho_{p,k}(t,X_{t}(x))\,\partial_{x}u_{p,k}(t,X_{t}(x)) \\
			&\geq -\rho_{p,k}(t,X_{t}(x))\left(\frac{a}{\mu}\rho_{p,k}^{\gamma}(t,X_{t}(x))+1\right)^{\frac{1}{p-1}} \\
			&\geq -\left(\frac{a}{\mu}\right)^{\frac{1}{p-1}}\rho_{p,k}^{1+\frac{\gamma}{p-1}}(t,X_{t}(x))-\rho_{p,k}(t,X_{t}(x)),
		\end{align*}
		where in the last line we used the inequality $(x+y)^{\alpha}\leq (x^{\alpha}+y^{\alpha})$ for $\alpha<1$. It follows that
		\begin{equation*}
			\partial_{t}\rho_{p,k}^{\frac{-\gamma}{p-1}}(t,X_{t}(x))
			\leq \frac{\gamma}{p-1}\left(\frac{a}{\mu}\right)^{\frac{1}{p-1}} + \frac{\gamma}{p-1}\rho_{p,k}^{\frac{-\gamma}{p-1}}(t,X_{t}(x)).
		\end{equation*}

        On the other hand, define $Y_{p,k}(s) = \rho_{p,k}^{\frac{-\gamma}{p-1}}(s,X_{s}(x))$. It holds that
		\begin{equation*}
			\partial_{s}Y_{p,k}(s)-\frac{\gamma}{p-1}Y_{p,k}(s)\leq \frac{\gamma}{p-1}\left(\frac{a}{\mu}\right)^{\frac{1}{p-1}},
		\end{equation*}
		which is equivalent to
		\begin{equation*}
			\partial_{s}\left(Y_{p,k}(s)\exp\left(-\frac{\gamma s}{p-1}\right)\right) \leq \frac{\gamma}{p-1}\left(\frac{a}{\mu}\right)^{1/(p-1)}\exp\left\{-\frac{\gamma s}{p-1}\right\}.
		\end{equation*}
		Integrating from $0$ to $t$ leads to
		\begin{equation*}
			Y_{p,k}(t)\exp\left(-\frac{\gamma t}{p-1}\right)-Y_{p,k}(0)\leq \left(\frac{a}{\mu}\right)^{1/(p-1)}\left[1 - \exp\left(-\frac{\gamma t}{p-1}\right)\right].
		\end{equation*}
		Hence, it follows that
		\begin{equation*}
			\rho_{p,k}^{\frac{\gamma}{p-1}}(t,X_{t}(x))\leq \left(\frac{a}{\mu}\right)^{\frac{1}{p-1}}\left(\exp\left(\frac{\gamma t}{p-1}\right)-1\right) + \rho_{p,k,0}^{-\frac{\gamma}{p-1}}(x)\exp\left(\frac{\gamma t}{p-1}\right),
		\end{equation*}
		and consequently
		\begin{equation*}
			\rho_{p,k}(t,X_{t}(x))\geq \frac{1}{\left(\rho_{0}^{-\frac{\gamma}{p-1}}e^{\frac{\gamma t}{p-1}} + \left(\frac{a}{\mu}\right)^{\frac{1}{p-1}}\left(e^{\frac{\gamma t}{p-1}} - 1\right)\right)^{\frac{p-1}{\gamma}}}
			= \frac{\rho_{0}e^{-t}}{\left(1 + \left(\frac{a}{\mu}\rho_{0}^{\gamma}\right)^{\frac{1}{p-1}}\left(1 - e^{-\frac{\gamma t}{p-1}}\right)\right)^{\frac{p-1}{\gamma}}}.
		\end{equation*}
		Using the inequality $1-e^{-x}\leq x$ for $x\geq 0$, we obtain
		\begin{equation*}
			\rho_{p,k}(t,X_{t}(x))\geq \frac{\rho_{0}e^{-t}}{\left(1 + \left(\frac{a}{\mu}\rho_{0}^{\gamma}\right)^{\frac{1}{p-1}}\frac{\gamma t}{p-1}\right)^{\frac{p-1}{\gamma}}}.
		\end{equation*}
		Note that
		\begin{align*}
			\left(1 + \left(\frac{a}{\mu}\rho_{0}^{\gamma}\right)^{\frac{1}{p-1}}\frac{\gamma}{p-1}t\right)^{\frac{p-1}{\gamma}}
			&\leq \max\left\{1,\left(\frac{a}{\mu}\|\rho_{0}\|_{L^{\infty}}^{\gamma}\right)^{1/\gamma}\right\}\left(1+\frac{\gamma t}{p-1}\right)^{\frac{p-1}{\gamma}} \\
			&\leq \max\left\{1,\left(\frac{a}{\mu}\right)^{1/\gamma}\|\rho_{0}\|_{L^{\infty}}\right\}e^{t},
		\end{align*}
		where in the second line we used the inequality $(1+x)^{\frac{1}{x}}\leq e$ for all $x>0$. Therefore,
		\begin{equation}\label{eq: low rho}
			\rho_{p,k}(t,X_{t}(x))\geq \frac{\rho_{0}e^{-2t}}{\max\left\{1,\left(\frac{a}{\mu}\right)^{1/\gamma}\|\rho_{0}\|_{L^{\infty}}\right\}}.
		\end{equation}
		Thus we obtain the lower bound.
		
		\noindent\textbf{Upper bound for the density.} The upper bound for the density is obtained by a potential function $\psi_{p,k}$ introduced by Basov and Shelukhin \cite{Basov1999} and used also in \cite[Proposition 2.1]{Bresch2026}. 
		
		For $x\in \Omega_k$, we define the potential $\psi_{p,k}$ as follows:		\begin{align}\label{def: psi}
			\psi_{p,k}(t, x)=&\int_0^t\left(\rho_{p,k} u_{p,k}^2-\mu\left|\partial_x u_{p,k}\right|^{p-2} \partial_x u_{p,k}+a\rho_{p,k}^\gamma\right)(s,x)\,\mathrm{d} s\nonumber\\
			&\qquad -\frac{1}{|\Omega_k|}\int_{\Omega_k}\left(\int_{y}^x\left(\rho_{p,k,0} u_{p,k,0}\right)(z)\,\mathrm{d} z\right)\,\mathrm{d}y.
		\end{align}
		From the momentum equation, we compute
		\begin{equation*}
			\begin{aligned}
				& \partial_t \psi_{p,k}(t, x)=\rho_{p,k} u_{p,k}^2-\mu\left|\partial_x u_{p,k}\right|^{p-2} \partial_x u_{p,k}+a\rho_{p,k}^\gamma, \\
				& \partial_x \psi_{p,k}(t, x)=\int_0^t\left(\rho_{p,k} u_{p,k}^2-\mu\left|\partial_x u_{p,k}\right|^{p-2} \partial_x u_{p,k}+ a\rho_{p,k}^\gamma\right)_x\,\mathrm{d} s-\left(\rho_{p,k,0} u_{p,k, 0}\right)(x)=-\rho_{p,k} u_{p,k}(t, x).
			\end{aligned}
		\end{equation*}
		Combining these two expressions above yields
		\begin{equation*}
			\partial_t \psi_{p,k}(t, x)+u_{p,k}(t, x) \partial_x \psi_{p,k}(t, x)=-\mu\left|\partial_x u_{p,k}\right|^{p-2} \partial_x u_{p,k}+a \rho_{p,k}^\gamma .
		\end{equation*}
		Moreover, 
		\begin{align*}
			\psi_{p,k}(t,x)=\frac{1}{|\Omega_k|}\int_{\Omega_k}\psi_{p,k}(t,y)\,\mathrm{d}y+\frac{1}{|\Omega_k|}\int_{\Omega_k}\left(\int_{y}^x\partial_z \psi_{p,k}(t,z)\,\mathrm{d} z\right)\,\mathrm{d}y,
		\end{align*}
		which gives
		\begin{align*}
			\psi_{p,k}(t,x)-\psi_{p,k}(0,x)=&\frac{1}{|\Omega_k|}\int_{\Omega_k}\int_0^t\left(\rho_{p,k} u_{p,k}^2-\mu\left|\partial_x u_{p,k}\right|^{p-2} \partial_x u_{p,k}+\rho_{p,k}^\gamma\right)(s,y)\,\mathrm{d} s\,\mathrm{d}y\\
			&-\frac{1}{|\Omega_k|}\int_{\Omega_k}\left(\int_{y}^x\rho_{p,k} u_{p,k}(t,z)\,\mathrm{d} z\right)\,\mathrm{d}y\\
			=:&\,\mathcal{I}_1+\mathcal{I}_2.
		\end{align*}
		The energy estimates in Proposition \ref{prop: energy} imply
		\begin{align*}
			|\mathcal{I}_1|\leq& \frac{2T}{|\Omega_k|}\int_{\Omega_k}\frac{1}{2}\rho_{p,k} u_{p,k}^2\,\mathrm{d}y+\frac{1}{|\Omega_k|}\int_0^t\int_{\Omega_k}\mu\left|\partial_x u_{p,k}\right|^{p-1}\,\mathrm{d}y \,\mathrm{d} s\\
			&+\frac{\gamma-1}{a}\frac{T}{|\Omega_k|}\int_{\Omega_k}\frac{a}{\gamma-1}\rho_{p,k}^\gamma \,\mathrm{d}y\\
			\leq& \left(\frac{2T}{|\Omega_k|}+\frac{(\gamma-1)T}{a|\Omega_k|}\right)E_0+\frac{1}{|\Omega_k|} \int_0^t\int_{\Omega_k}\mu\left(\frac{p-1}{p}\left|\partial_x u_{p,k}\right|^{p}+\frac{1}{p}\right)\,\mathrm{d}y \,\mathrm{d} s\\
			\leq&\left(\frac{2T}{|\Omega_k|}+\frac{(\gamma-1)T}{a|\Omega_k|}+\frac{\mu}{|\Omega_k|}\right)E_0+\frac{\mu}{p}t ,
		\end{align*}
		where we used Young's inequality $xy\leq \frac{p-1}{p}x^{\frac{p}{p-1}}+\frac{1}{p}y^p$ for $x,y\geq 0$.
		
		Similarly, we bound $\mathcal{I}_2$ by
		\begin{align*}
			\left|-\int_{y}^{x}\rho_{p,k}u_{p,k}(t,z)\,\mathrm{d}z\right|\leq& \int_{\Omega_k}\left|\rho_{p,k}u_{p,k}(t,z)\right|\,\mathrm{d}z\\ \leq& \left(\int_{\Omega_k}\rho_{p,k}\,\mathrm{d}z\right)^{\frac{1}{2}}\left(\int_{\Omega_k}\rho_{p,k}u_{p,k}^2\,\mathrm{d}z\right)^{\frac{1}{2}} \\
			\leq &\left(\int_{\Omega_k}\rho_{p,k}^{\gamma}\,\mathrm{d}x\right)^{\frac{1}{\gamma}}|\Omega_k|^{1-\frac{1}{\gamma}} \left(2E_{0}\right)^{\frac{1}{2}}\\
			\leq &\left(\frac{\gamma-1}{a}E_0\right)^{\frac{1}{\gamma}}\left(2E_{0}\right)^{\frac{1}{2}}|\Omega_k|^{1-\frac{1}{\gamma}}\\
			\leq &\sqrt{2}\left(\frac{\gamma-1}{a}\right)^{\frac{1}{\gamma}}|\Omega_k|^{1-\frac{1}{\gamma}}E_0^{\frac{1}{\gamma}+\frac{1}{2}}.
		\end{align*}
		Thus
		\begin{equation}\label{es: psi diff}
			\begin{aligned}
				|\psi_{p,k}(t,x)-\psi_{p,k}(0,x)|\leq& \left(\frac{2T}{|\Omega_k|}+\frac{(\gamma-1)T}{a|\Omega_k|}
				+\frac{\mu}{|\Omega_k|}\right)E_0\\
				&\qquad+\sqrt{2}\left(\frac{\gamma-1}{a}\right)^{\frac{1}{\gamma}}|\Omega_k|^{1-\frac{1}{\gamma}}E_0^{\frac{1}{\gamma}+\frac{1}{2}}+\frac{\mu}{p}t\\
				\leq & C(T,a,\mu,\gamma,k,E_0)+\frac{\mu}{2}t.
			\end{aligned}
		\end{equation}

        Then, notice that 		\begin{equation*}
			\begin{aligned}
				\partial_t&\left(\rho_{p,k}^{\mu}\exp\{-\psi_{p,k}\}\right)+u_{p,k}\partial_x\left(\rho_{p,k}^{\mu}\exp\{-\psi_{p,k}\}\right)\\&=\rho_{p,k}^{\mu}\exp\{-\psi_{p,k}\}\left(-\mu\partial_x u_{p,k}+\mu\left|\partial_x u_{p,k}\right|^{p-2}\partial_x u_{p,k}-a\rho_{p,k}^{\gamma}\right).
			\end{aligned}
		\end{equation*}
		Here the exponent $\mu$ in $\rho_{p,k}^{\mu}$ is the viscosity coefficient.
		For $\partial_x u_{p,k}>0$, we have
		\begin{equation*}
			-\mu \partial_x u_{p,k}+\mu\left|\partial_x u_{p,k}\right|^{p-2} \partial_x u_{p,k}-a \rho_{p,k}^\gamma=\sigma_{p,k}-\mu \partial_x u_{p,k} \leq \mu,
		\end{equation*}
		while for $\partial_x u_{p,k}<0$, since the function $z\mapsto z(|z|^{p-2}-1)$ is bounded by 1 for $z<0$, we have
		\begin{equation*}
			-\mu \partial_x u_{p,k}+\mu\left|\partial_x u_{p,k}\right|^{p-2} \partial_x u_{p,k}-\rho_{p,k}^\gamma \leq \mu \partial_x u_{p,k}\left(\left|\partial_x u_{p,k}\right|^{p-2}-1\right) \leq \mu.
		\end{equation*}
		Therefore,		\begin{equation*}
			\partial_t\left(\rho_{p,k}^{\mu}\exp\{-\psi_{p,k}\}\right)+u_{p,k}\partial_x\left(\rho_{p,k}^{\mu}\exp\{-\psi_{p,k}\}\right)\leq \mu\rho_{p,k}^{\mu}\exp\{-\psi_{p,k}\}.
		\end{equation*}
		Applying the maximum principle and Gr\"{o}nwall's lemma, we deduce that
		\begin{equation*}
			\rho_{p,k}(t,x)\leq \rho_{p,k,0}\exp\left[\frac{1}{\mu}\left(\psi_{p,k}(t,x)-\psi_{p,k}(0,x)\right)+t\right].
		\end{equation*}
		Using the difference estimates \eqref{es: psi diff} for $\psi_{p,k}$, we obtain the upper bound on $\Omega_k$
		\begin{equation}\label{up rho_pk}
			\rho_{p,k}\leq \sup\limits_{\Omega_k}\rho_{p,k,0}(x)\exp\left[C(T,a,\mu,\gamma,k,E_0)+\frac{3}{2}t\right].
		\end{equation}
		Thus, the density remains bounded from  above uniformly in $p$ on $\Omega_k$.
	\end{proof}
	
	\begin{proof}[Proof of Proposition \ref{prop: rho dotu}]
		The momentum equation reads
		\begin{equation*}
			\rho_{p,k}\dot{u}_{p,k} - \mu\partial_{x}\left(|\partial_{x}u_{p,k}|^{p-2}\partial_{x}u_{p,k}\right) + a\partial_{x}\rho_{p,k}^{\gamma}=0.
		\end{equation*}
		Multiplying by $\dot{u}_{p,k}$ and integrating over $\Omega_k$, we get
		\begin{align*}
			\int_{\Omega_k}\rho_{p,k}|\dot{u}_{p,k}|^{2}\,\mathrm{d}x&+\mu\frac{\mathrm{d}}{\mathrm{d}t}\int_{\Omega_k}\frac{|\partial_{x}u_{p,k}|^{p}}{p}\,\mathrm{d}x
			+\mu\int_{\Omega_k}|\partial_{x}u_{p,k}|^{p-2}\partial_{x}u_{p,k}\,\partial_{x}(u_{p,k}\partial_{x}u_{p,k})\,\mathrm{d}x\\&- a\int_{\Omega_k}\rho_{p,k}^{\gamma}\partial^{2}_{tx}u_{p,k}\,\mathrm{d}x+ a\int_{\Omega_k}\partial_{x}(\rho_{p,k})^{\gamma}\,u_{p,k}\partial_{x}u_{p,k}\,\mathrm{d}x = 0.
		\end{align*}
		First note that
		\begin{equation*}
			\partial_x\left(|\partial_x u_{p,k}|^p\right)=p|\partial_x u_{p,k}|^{p-2}\partial_x u_{p,k}\partial_x^2 u_{p,k}.
		\end{equation*}
		Using the boundary condition $\left.u\right|_{\partial \Omega_k}=0$ and integration by parts, we get
		\begin{align*}
			\mu\int_{\Omega_k}|\partial_{x}u_{p,k}|^{p-2}\partial_{x}u_{p,k}\,\partial_{x}(u_{p,k}\partial_{x}u_{p,k})\,\mathrm{d}x
			=&\mu\int_{\Omega_k}|\partial_{x}u_{p,k}|^{p}\partial_{x}u_{p,k}\,\mathrm{d}x\\
			&+\mu\int_{\Omega_k}|\partial_{x}u_{p,k}|^{p-2}\partial_{x}u_{p,k}\,\partial_{x}^{2}u_{p,k}\,u_{p,k}\,\mathrm{d}x\\
			=&\mu\left(1-\frac{1}{p}\right)\int_{\Omega_k}|\partial_{x}u_{p,k}|^{p}\partial_{x}u_{p,k}\,\mathrm{d}x\\
			=&\left(1-\frac{1}{p}\right)\int_{\Omega_k}|\partial_{x}u_{p,k}|^{2}\left(\sigma_{p,k}+a\rho_{p,k}^{\gamma}\right)\,\mathrm{d}x.
		\end{align*}
		Second, observe that
		\begin{equation*}
			-\int_{\Omega_k}\rho_{p,k}^{\gamma}\partial^{2}_{tx}u_{p,k}\,\mathrm{d}x
			=-\frac{\mathrm{d}}{\mathrm{d}t}\int_{\Omega_k}\rho_{p,k}^{\gamma}\partial_{x}u_{p,k}\,\mathrm{d}x
			+\int_{\Omega_k}\partial_{t}(\rho_{p,k}^{\gamma})\partial_{x}u_{p,k}\,\mathrm{d}x
		\end{equation*}
		and, by the continuity equation $\eqref{PDE,pk}_1$
		\begin{equation*}
			\partial_{t}(\rho_{p,k}^{\gamma}) + u_{p,k}\partial_{x}(\rho_{p,k}^{\gamma}) + \gamma\rho_{p,k}^{\gamma}\partial_{x}u_{p,k}=0,
		\end{equation*}
		one obtains that	\begin{equation*}
			\int_{\Omega_k}\partial_{t}\rho_{p,k}^{\gamma}\partial_{x}u_{p,k}\,\mathrm{d}x
			+\int_{\Omega_k}u_{p,k}\partial_{x}(\rho_{p,k}^{\gamma})\partial_{x}u_{p,k}\,\mathrm{d}x
			=-\int_{\Omega_k}\gamma\rho_{p,k}^{\gamma}|\partial_{x}u_{p,k}|^{2}\,\mathrm{d}x.
		\end{equation*}
		Hence,
		\begin{equation*}
			-\int_{\Omega_k}\rho_{p,k}^{\gamma}\partial^{2}_{tx}u_{p,k}\,\mathrm{d}x
			+\int_{\Omega_k}\partial_{x}(\rho_{p,k})^{\gamma}u_{p,k}\partial_{x}u_{p,k}\,\mathrm{d}x
			=-\frac{\mathrm{d}}{\mathrm{d}t}\int_{\Omega_k}\rho_{p,k}^{\gamma}\partial_{x}u_{p,k}\,\mathrm{d}x
			-\int_{\Omega_k}\gamma\rho_{p,k}^{\gamma}|\partial_{x}u_{p,k}|^{2}\,\mathrm{d}x.
		\end{equation*}
		Putting together the previous computations, we deduce that
		\begin{align*}
			\int_{\Omega_k}\rho_{p,k}|\dot{u}_{p,k}|^{2}\,\mathrm{d}x
			+&\mu\frac{\mathrm{d}}{\mathrm{d}t}\int_{\Omega_k}\frac{|\partial_{x}u_{p,k}|^{p}}{p}\,\mathrm{d}x
			+a\left(1-\frac{1}{p}\right)\int_{\Omega_k}|\partial_{x}u_{p,k}|^{2}\rho_{p,k}^{\gamma}\,\mathrm{d}x
			\\=& \frac{\mathrm{d}}{\mathrm{d}t}\int_{\Omega_k}a\rho_{p,k}^{\gamma}\partial_{x}u_{p,k}\,\mathrm{d}x
			+a\int_{\Omega_k}\gamma\rho_{p,k}^{\gamma}|\partial_{x}u_{p,k}|^{2}\,\mathrm{d}x\\
			&-\left(1-\frac{1}{p}\right)\int_{\Omega_k}|\partial_{x}u_{p,k}|^{2}\sigma_{p,k}\,\mathrm{d}x.
		\end{align*}
		Thus, integrating in time from $0$ to $t\in(0,T)$, we infer that
		\begin{align*}
			&\int_{0}^{t}\int_{\Omega_k}\rho_{p,k}|\dot{u}_{p}|^{2}\,\mathrm{d}x \,\mathrm{d}s
			+ \mu\int_{\Omega_k}\frac{|\partial_{x}u_{p,k}|^{p}}{p}(t,x)\,\mathrm{d}x
			+a\left(1-\frac{1}{p}\right)\int_{0}^{t}\int_{\Omega_k}|\partial_{x}u_{p,k}|^{2}\rho_{p,k}^{\gamma}\,\mathrm{d}x \,\mathrm{d}s \\
			&= \int_{\Omega_k}a\rho_{p,k}^{\gamma}\partial_{x}u_{p,k}(t,x)\,\mathrm{d}x
			- \int_{\Omega_k}a\rho_{p,k}^{\gamma}\partial_{x}u_{p,k}(0)\,\mathrm{d}x
			+a\int_{0}^{t}\int_{\Omega_k}\gamma\rho_{p,k}^{\gamma}|\partial_{x}u_{p,k}|^{2}\,\mathrm{d}x\,\mathrm{d}s \\
			&\quad-\left(1-\frac{1}{p}\right)\int_{0}^{t}\int_{\Omega_k}|\partial_{x}u_{p,k}|^{2}\sigma_{p,k}\,\mathrm{d}x \,\mathrm{d}s+ \mu\int_{\Omega_k}\frac{|\partial_{x}u_{p,k}(0)|^{p}}{p}\,\mathrm{d}x =: \sum_{i=1}^{5} I_{i}.
		\end{align*}
		Concerning $I_{1}$, for all $t\in(0,T]$, using Young's inequality $xy\leq \frac{1}{p}x^{p}+\frac{p-1}{p}y^\frac{p}{p-1}$ for $x,y\geq 0$ and energy estimates in Proposition \ref{prop: energy}, we have
		\begin{align*}
			I_1\leq& \frac{p-1}{p}\left(\frac{2}{\mu}\right)^{\frac{1}{p-1}}\int_{\Omega_k}\rho_{p,k}^{\gamma\cdot \frac{p}{p-1}}(t,x)\,\mathrm{d}x+ \frac{\mu}{2p}\int_{\Omega_k}\left|\partial_{x}u_{p,k}\right|^{p}(t,x)\,\mathrm{d}x\\
			\leq&\left\|\rho_{p,k}\right\|_{L_x^{\infty}}^{\frac{\gamma}{p-1}}\frac{p-1}{p}\left(\frac{2}{\mu}\right)^{\frac{1}{p-1}}\int_{\Omega_k}\rho_{p,k}^{\gamma}(t,x)\,\mathrm{d}x+ \frac{\mu}{2}\int_{\Omega_k}\frac{|\partial_{x}u_{p,k}|^{p}}{p}(t,x)\,\mathrm{d}x\\
			\leq&C(T,\gamma,\mu,E_0,k)+\frac{\mu}{2}\int_{\Omega_k}\frac{|\partial_{x}u_{p,k}|^{p}}{p}(t,x)\,\mathrm{d}x.
		\end{align*}
		Similarly, for $I_3$, using Young's inequality $xy\leq \frac{2}{p}x^{\frac{p}{2}}+\frac{p-2}{p}y^\frac{p}{p-2}$ for $x,y\geq 0$ and energy estimates in Proposition \ref{prop: energy}, we derive
		\begin{align*}
			I_3
			\leq& \left\|\rho_{p,k}\right\|_{L_x^{\infty}}^{\frac{2\gamma}{p}}\left[\frac{p-2}{p}\int_0^t \int_{\Omega_k}\rho_{p,k}^{\frac{p-2}{p}\cdot\gamma\cdot \frac{p}{p-2}}\,\mathrm{d}x \,\mathrm{d}s+ 2\int_0^t \int_{\Omega_k}\frac{\left|\partial_{x}u_{p,k}\right|^{p}}{p}\,\mathrm{d}x \,\mathrm{d}s\right]\\
			\leq& C(T,\gamma,\mu,E_0,k)+C(T,\gamma,\mu,E_0,k)\int_0^t \int_{\Omega_k}\frac{\left|\partial_{x}u_{p,k}\right|^{p}}{p}\,\mathrm{d}x \,\mathrm{d}s.
		\end{align*}
		For $I_4$, we recall that $\partial_x\sigma_{p,k}=\rho_{p,k}\dot{u}_{p,k}$. Since $\int_{\Omega_k}\partial_x u_{p,k}(t,y)\,\mathrm{d}y=0$, the mean value theorem implies that for each $t$ there exists $x(t)\in\Omega_k$ such that $\partial_x u_{p,k}(t,x(t))=0$. Integrating the momentum equation from $x(t)$ to $x$ gives
		\begin{equation}\label{def sigma1}
			\sigma_{p,k}(t,x) = \int_{x(t)}^{x} (\rho_{p,k}\dot{u}_{p,k})(t,z)\,\mathrm{d}z - a\rho_{p,k}^{\gamma}(t,x(t)).
		\end{equation}
		For $p\geq 4$, using the H\"{o}lder's inequality, the upper bound for density \eqref{es: bound rho_pk} and energy estimates, we obtain 
		\begin{align*}
			I_4
			\leq& \left(1-\frac{1}{p}\right)\int_{0}^{t}\int_{\Omega_k}\left|\int_{x(s)}^{x} (\rho_{p,k}\dot{u}_{p,k})(s,z)\,\mathrm{d}z - a\rho_{p,k}^{\gamma}(s,x(s))\right||\partial_{x}u_{p,k}|^{2}\,\mathrm{d}x \,\mathrm{d}s\\
			\leq& \left(1-\frac{1}{p}\right) \int_{0}^{t}\int_{\Omega_k}\left(\int_{\Omega_k} \left|\rho_{p,k}\dot{u}_{p,k}\right|(s,z)\,\mathrm{d}z+a\left\|\rho_{p,k}\right\|_{L_x^{\infty}}^{\gamma}\right)|\partial_{x}u_{p,k}|^{2}\,\mathrm{d}x \,\mathrm{d}s\\
			\leq& \left(1-\frac{1}{p}\right) \int_{0}^{t}\left(\int_{\Omega_k} \left|\rho_{p,k}\dot{u}_{p,k}\right|(s,z)\,\mathrm{d}z\right)\int_{\Omega_k} |\partial_{x}u_{p,k}|^{2}\,\mathrm{d}x \,\mathrm{d}s\\
			&+a\left(1-\frac{1}{p}\right) \int_{0}^{t}\int_{\Omega_k}\left\|\rho_{p,k}\right\|_{L_x^{\infty}}^{\gamma}|\partial_{x}u_{p,k}|^{2}\,\mathrm{d}x \,\mathrm{d}s\\
			\leq & \left(1-\frac{1}{p}\right)\int_{0}^{t}\left(\int_{\Omega_k}\rho_{p,k}\,\mathrm{d}z\right)^{\frac{1}{2}}\left(\int_{\Omega_k}\rho_{p,k}|\dot{u}_{p,k}|^2\,\mathrm{d}z\right)^{\frac{1}{2}}\int_{\Omega_k} |\partial_{x}u_{p,k}|^{2}\,\mathrm{d}x \,\mathrm{d}s\\
			&+a\left(1-\frac{1}{p}\right)\left\|\rho_{p,k}\right\|_{L_x^{\infty}}^{\gamma}\int_{0}^{t}\int_{\Omega_k}|\partial_{x}u_{p,k}|^{p}\,\mathrm{d}x \,\mathrm{d}s\\
			\leq & \left(1-\frac{1}{p}\right)\int_{0}^{t}\left[\frac{1}{2}\left(\int_{\Omega_k}\rho_{p,k}|\dot{u}_{p,k}|^2\,\mathrm{d}z\right)+\frac{1}{2}\left(\left\|\rho_{p,k}\right\|_{L_x^{\infty}}^{\frac{1}{2}}|\Omega_k|^{\frac{1}{2}}\int_{\Omega_k} |\partial_{x}u_{p,k}|^{2}\,\mathrm{d}x\right)^2\right] \,\mathrm{d}s\\
			&+C(T,\mu,a,\gamma,E_0,k)\\
			\leq & \int_{0}^{t}\left[\frac{1}{2}\left(\int_{\Omega_k}\rho_{p,k}|\dot{u}_{p,k}|^2\,\mathrm{d}z\right)+\frac{1}{2}\left\|\rho_{p,k}\right\|_{L_x^{\infty}}|\Omega_k|\left(\int_{\Omega_k} |\partial_{x}u_{p,k}|^{2}\,\mathrm{d}x\right)^2\right] \,\mathrm{d}s\\
			&+C(T,\mu,a,\gamma,E_0,k)\\
			\leq
			&\frac{1}{2}\int_{0}^{t}\int_{\Omega_k}\rho_{p,k}|\dot{u}_{p,k}|^2\,\mathrm{d}x \,\mathrm{d}s+\frac{1}{2}\left\|\rho_{p,k}\right\|_{L_x^{\infty}}|\Omega_k|^2\int_{0}^{t}\int_{\Omega_k}|\partial_x u_{p,k}|^{4}\,\mathrm{d}x \,\mathrm{d}s+C(T,\mu,a,\gamma,E_0,k)\\
			\leq &\frac{1}{2}\int_{0}^{t}\int_{\Omega_k}\rho_{p,k}|\dot{u}_{p,k}|^2\,\mathrm{d}x \,\mathrm{d}s+\frac{1}{2}\left\|\rho_{p,k}\right\|_{L_x^{\infty}}|\Omega_k|^2\int_{0}^{t}\int_{\Omega_k}(1+|\partial_x u_{p,k}|^{p})\,\mathrm{d}x \,\mathrm{d}s\\
			&+C(T,\mu,a,\gamma,E_0,k)\\
			\leq& \frac{1}{2}\int_{0}^{t}\int_{\Omega_k}\rho_{p,k}|\dot{u}_{p,k}|^2\,\mathrm{d}x \,\mathrm{d}s+ C(T,\mu,a,\gamma,E_0,k),
		\end{align*}
		where we used the inequality $|s|^4\leq 1+|s|^p$ for $p\geq 4$.
		Denoting
		\begin{equation*}
			Y(t) = \frac{1}{2}\int_{0}^{t}\int_{\Omega_k}\rho_{p,k}|\dot{u}_{p,k}|^{2}\,\mathrm{d}x\,\mathrm{d}s + \frac{\mu}{2}\int_{\Omega_k}\frac{|\partial_{x}u_{p,k}|^{p}}{p}(t,x)\,\mathrm{d}x,
		\end{equation*}
		we obtain that
		\begin{equation*}
			Y(t) \leq C(T,\gamma,\mu,a, E_0,k) +C(T,\gamma,\mu,a, E_0,k) \int_{0}^{t}Y(s)\,\mathrm{d}s. 
		\end{equation*}
		From here, we may conclude the proof via Gr\"{o}nwall's inequality.
	\end{proof}
	\begin{proof}[Proof of Proposition \ref{prop: bound tau_pk}]
		Recall the definition of $\sigma_{p,k}$ \eqref{def sigma1} and $x(t)$ in the proof of Proposition \ref{prop: rho dotu} (see the identity~\eqref{def sigma1}). Integrating the momentum equation from $x(t)$ to $x$, we get
		\begin{equation*}
			\mu |\partial_{x}u_{p,k}|^{p-2}\partial_{x}u_{p,k}(t,x) = a\rho_{p,k}^{\gamma}(t,x) + \int_{x(t)}^{x} (\rho_{p,k}\dot{u}_{p,k})(t,z)\,\mathrm{d}z - a\rho_{p,k}^{\gamma}(t,x(t)).
		\end{equation*}
		Then, using the bounds on $\rho_{p,k}$ and $\dot{u}_{p,k}$, we deduce that
		\begin{align*}
			\left|\int_{x(t)}^x(\rho_{p,k}\dot{u}_{p,k})(t,z)\,\mathrm{d}z\right|\leq& \int_{\Omega_k}\left|\rho_{p,k}\dot{u}_{p,k}\right|(t,z)\,\mathrm{d}z\\
			\leq &\left(\int_{\Omega_k}\rho_{p,k}\,\mathrm{d}z\right)^{\frac{1}{2}}\left(\int_{\Omega_k}\rho_{p,k}\left|\dot{u}_{p,k}\right|^2\,\mathrm{d}z\right)^{\frac{1}{2}}\\
			\leq &C(T,\mu,E_0,k).
		\end{align*}
		Therefore, using the upper bound of $\rho_{p,k}$ and Proposition \ref{prop: rho dotu}, we have
		\begin{align*}
			\int_0^T \left\||\partial_{x}u_{p,k}|^{p-2}\partial_{x}u_{p,k}\right\|_{L_x^{\infty}}^2\,\mathrm{d}t\leq &\int_0^T\left(a\left\|\rho_{p,k}^{\gamma}\right\|_{L_x^{\infty}}+\left(\int_{\Omega_k}\rho_{p,k}\,\mathrm{d}z\right)^{\frac{1}{2}}\left(\int_{\Omega_k}\rho_{p,k}\left|\dot{u}_{p,k}\right|^2\,\mathrm{d}z\right)^{\frac{1}{2}}\right)^2\,\mathrm{d}t\\
			\leq &C(T,\mu,E_0,k).
		\end{align*}
		Thus $|\partial_{x}u_{p,k}|^{p-2}\partial_{x}u_{p,k}$ belongs to $L^2(0,T;L^{\infty}(\Omega_k))$ uniformly in $p$, which completes the proof.
	\end{proof}

\bigskip

\noindent
{\bf Acknowledgement}. SL is supported by NSFC Projects 12331008 $\&$ 12411530065, the Young Elite Scientists Sponsorship Program by CAST 2023QNRC001, National Key Research $\&$ Development Programs 2023YFA1010900 and 2024YFA1014900, Shanghai Rising-Star Program 24QA2703600, Shanghai Qi-Guang Scholarship, and Shanghai Frontiers Science Center of Modern Analysis.

\medskip
\noindent
{\bf Statement of competing interests}. We declare that there are no conflicts of interest involved.

\medskip
\noindent
{\bf Data Availability Statement}. We declare that no data are associated with this work.

\medskip
\noindent
{\bf AI Statement}.  The authors thank the ChatGPT 5.5 AI model for editing and  fruitful discussions on the convergence of $\{\rho_p\}$ in $C(0,T; L^p(\R))$. All mathematical derivations, conclusions, and errors remain solely the responsibility of the authors.

\end{document}